\documentclass{article}
\usepackage[english]{babel}

\usepackage{graphicx}
\usepackage[pdfborder={0 0 0}]{hyperref}
\usepackage[utf8]{inputenc}

\usepackage{geometry}
\usepackage{mathtools}
\usepackage{amssymb}
\usepackage{dsfont}
\usepackage{amsmath}
\usepackage{url}
\usepackage{float}
\usepackage{comment}
\usepackage[T1]{fontenc}
\usepackage{bm}
\usepackage{graphicx}
\usepackage[normalem]{ulem}
\usepackage{commath}
\usepackage{amsthm}
\usepackage{mathrsfs}  
\usepackage{amsmath,amscd}
\usepackage{hyperref}
\usepackage{bm}
\usepackage{array}
\usepackage[nottoc,notlot,notlof]{tocbibind}
\usepackage[title]{appendix}
\usepackage{tikz}
\usepackage[title]{appendix}
\usepackage{xcolor}
\usepackage{leftidx}
\usepackage{stmaryrd}

\usetikzlibrary{fadings}
\usetikzlibrary{decorations.shapes}

\usepackage{pgfplots}

\usepackage{cleveref}

\usepackage{tikz-cd}
\usetikzlibrary{matrix,arrows}
\usepackage{subfiles}

\newcommand{\CR}{\overline{\partial}}
\newcommand{\N}{\mathbb{N}}
\newcommand{\Z}{\mathbb{Z}}

\newcommand{\C}{\mathbb{C}}
\newcommand{\R}{\mathbb{R}}
\newcommand{\T}{\mathbb{T}}

\newcommand{\Hil}{\mathbb{H}}

\newcommand{\B}{\mathcal{B}}
\newcommand{\E}{\mathcal{E}}
\newcommand{\Id}{\mathds{1}}

\newcommand{\Span}{\mathrm{Span}}

\newcommand{\Diag}{\mathrm{diag}}

\newcommand{\lhk}{\left(}
\newcommand{\rhk}{\right)}

\newcommand{\opp}{\mathcal{O}}

\newcommand{\M}{\mathcal{M}}
\renewcommand{\P}{\mathcal{P}}
\newcommand{\ind}{\mathrm{ind}}

\newcommand{\coker}{\mathrm{coker}}

\newcommand{\prt}{\mathrm{part}}
\newcommand{\inter}{\mathrm{inter}}

\newcommand{\op}[1]{\left\| #1 \right\|}

\makeatletter
\newcommand{\opnorm}{\@ifstar\@opnorms\@opnorm}
\newcommand{\@opnorms}[1]{%
  \left|\mkern-1.5mu\left|\mkern-1.5mu\left|
   #1
  \right|\mkern-1.5mu\right|\mkern-1.5mu\right|
}
\newcommand{\@opnorm}[2][]{%
  \mathopen{#1|\mkern-1.5mu#1|\mkern-1.5mu#1|}
  #2
  \mathclose{#1|\mkern-1.5mu#1|\mkern-1.5mu#1|}
}
\makeatother

\pgfplotsset{compat=1.10}
\usepgfplotslibrary{fillbetween}
\usetikzlibrary{patterns}

\newcommand{\wt}[1]{\widetilde{#1}}

\def\[#1\]{\begin{align*}#1 \end{align*}}

\newcommand{\vierkanttwee}[4]{\lhk \begin{matrix} #1&#2\cr#3&#4\end{matrix}\rhk}

\DeclarePairedDelimiter\floor{\lfloor}{\rfloor}

\newcommand{\pa}[1]{\left( #1 \right)}

\newtheorem{thm}{Theorem}[section]
\newtheorem{definition}[thm]{Definition}
\newtheorem{theorem}[thm]{Theorem}
\newtheorem{proposition}[thm]{Proposition}
\newtheorem{corollary}[thm]{Corollary}
\newtheorem{lemma}[thm]{Lemma}
\newtheorem{conjecture}[thm]{Conjecture}
\newtheorem{example}[thm]{Example}

\title{Floer homology for Hamiltonian PDEs: Fredholm theory}
\author{Oliver Fabert \\ Niek Lamoree}
\date{ }

\begin{document}
\maketitle

\tikzset{decorate sep/.style 2 args=
{decorate,decoration={shape backgrounds,shape=circle,shape size=#1,shape sep=#2}}}

\abstract{By coupling a Hamiltonian mechanical system with a linear Hamiltonian field theory one obtains an infinite-dimensional Hamiltonian system with regularizing nonlinearity, where the underlying phase space is given by the product of a finite-dimensional symplectic manifold with an infinite-dimensional linear symplectic Hilbert space. After our compactness results in \cite{FL}, \cite{FL2} we continue our program for defining a Floer homology theory for this class of infinite-dimensional Hamiltonian systems. Based on the presence of small divisors we introduce a new notion of nondegeneracy for time-periodic solutions which allows us to prove that the linearization of the nonlinear Floer operator is Fredholm when viewed as a map between suitable Sobolev space completions. 

\tableofcontents

\section{Motivation and summary}\label{motivation-section}

Classical physics consists of two parts, classical mechanics and classical field theory. While in classical mechanics one studies the dynamics of particles, in classical field theory one studies the dynamics of fields. In both cases it is well-known that problems can be studied as Hamiltonian systems on appropriate phase spaces.\par
The phase spaces of classical mechanics are finite-dimensional symplectic manifolds $M=(M,\omega_M)$, where the symplectic form $\omega_M$ is a closed, nondegenerate two-form. Since $\omega_M$ provides a one-to-one correspondence between vector fields and one-forms on $M$, it assigns a time-dependent (Hamiltonian) vector field $X^H_t$ on $M$ to each time-dependent (Hamiltonian) function $H_t:M\to\R$ which represents the energy of the underlying mechanical system. Assuming that the Hamiltonian function $H_t$ is time-periodic with period $T>0$, $H_{t+T}=H_t$, one is interested in the set $\P(H)$ of $T$-periodic orbits of the time-periodic Hamiltonian vector field $X^H_{t+T}=X^H_t$. In his famous paper \cite{Gr} M. Gromov introduced moduli spaces of maps from a Riemann surface to the symplectic manifold $M$ equipped with an $\omega_M$-compatible almost complex structure $J_M:TM\to TM$, $J_M^2=-1$ satisfying the Cauchy-Riemann equation with a zeroth order Hamiltonian term, now called Floer curves. In the case when $M$ is closed and $\pi_2(M)=\{0\}$, he used the fact that the perturbed Cauchy-Riemann operator is nonlinear Fredholm as well as compactness results to prove that $\P(H)$ is non-empty. Building on Gromov's work, A. Floer extended the study of the above moduli spaces of maps satisfying a perturbed Cauchy-Riemann equation to develop the tool of Floer homology in \cite{Fl}. He used it to prove the homological Arnold conjecture, which in the case of $\Z_2$-coefficients states that number of $T$-periodic orbits is bounded from below by the sum of the $\Z_2$-Betti numbers of the singular homology of $M$, $\#P(H)\geq\dim_{\Z_2} H_*(M,\Z_2)$, provided that all elements in $\P(H)$ are nondegenerate in the sense that the linearized return map has no eigenvalue equal to one. On the other hand, the second important case are cotangent bundles, $(M,\omega_M)=(T^*Q,d\lambda_M)$, where $Q$ is a closed manifold and $\lambda_M$ denotes the Liouville one-form on $T^*Q$. Note that a choice of Riemannian metric on $Q$ leads to a natural choice of Riemannian metric $\langle\cdot,\cdot\rangle_M$ on $T^*Q$, defined using the Levi-Civita connection on $T^*M\cong TM$, with corresponding $\omega_M$-compatible almost complex structure $J_M$. By adapting Floer's proof, in particular, establishing $C^0$-bounds for Floer curves, it was shown e.g. in \cite{AS} that $\#\P(H)\geq \dim_{\Z_2} H_*(\Lambda Q,\Z_2)$ with $\Lambda Q=C^0(S^1,Q)$ denoting the based loop space of $Q$, provided not only that all orbits in $\P(H)$ are nondegenerate but also that $H$ is asymptotically quadratic with respect to the momentum coordinates $p$ in the sense that 
\begin{itemize}
\item $dH_t(q,p)\cdot p\frac{\partial}{\partial p} - H_t(q,p)\geq c_0 |p|^2 - c_1$, for some constants $c_0>0$ and $c_1\geq 0$,
\item $|\nabla_q H_t(q,p)|\leq c_2 (1+|p|^2),\,\,|\nabla_p H_t(q,p)|\leq c_2 (1+|p|)$ for some constant $c_2\geq 0$.
\end{itemize}
As a simple example of a problem studied in classical mechanics, let us consider the dynamics of a charged particle with locus $q(t)$ constrained to a submanifold $Q\subset\R^N$ and shape function $\rho$ , driven by a force field with underlying time-dependent scalar potential field $\varphi(t,x)$ on $\R^N$ 
\[ \nabla_t^2 q(t) = -\nabla(\varphi*\rho)(t,q(t)),\] where $\nabla$ denotes the gradient for functions on $Q\subset\R^N$ as well as the Levi-Civita connection on the tangent bundle $TQ$ of $Q$ with respect to the induced canonical Riemannian metric, and $*$ denotes convolution with respect to the space coordinates in $\R^N$. Assuming that the field is time-periodic with period $T$, $\varphi(t+T,x)=\varphi(t,x)$, one is interested in lower bounds for the number of $T$-periodic solutions $q(t+T)=q(t)$. It is well-known that this problem can be formulated as a problem about $T$-periodic solutions $u=(q,p)$ of a Hamilton's equations for the Hamiltonian function $$H_t: T^*Q\to\R,\,\, H_t(q,p)=\frac{1}{2}|p|^2+(\varphi*\rho)(t,q).$$\par 

In the opposite direction, in classical field theory one can study the dynamics of the scalar potential field $\varphi(t,x)$ driven by the known dynamics of the particle with locus $q(t)$ and shape function $\rho$, 
\[ \partial_t^2 \varphi(t,x) = \Delta\varphi(t,x)-\varphi(t,x)-\rho(x-q(t)),\] where $\Delta=\partial_{x_1}^2+\ldots+\partial_{x_N}^2$ denotes the Laplacian for functions on $\R^N$. Here and below we will only be interested in solutions which are $2\pi$-periodic with respect to the space coordinates, so that for every $t\in\R$ we may view $\varphi(t,\cdot)$ as a map from the torus $\T^N=\R^N/(2\pi\Z)^N$ to $\R$, and further view $Q$ as a submanifold of $\T^N$. Assuming now that the motion of the particle is time-periodic, $q(t+T)=q(t)$, it clearly makes sense to ask for solutions for the field which are $T$-periodic with respect to time and $2\pi$-periodic with respect to space, $\varphi(t+T,x)=\varphi(t,x)=\varphi(t,x+2\pi)$.\par Assuming that neither the time-dependent locus $q(t)$ of the particle nor the time-dependent scalar potential field $\varphi(t,x)$ are given a priori, we arrive at the coupled particle-field systems as studied e.g. in \cite{BaGa}, \cite{Kun}, \cite{Spo}. In this case we rather look for time- and space-periodic solutions $q(t+T)=q(t)$, $\varphi(t+T,x)=\varphi(t,x)=\varphi(t,x+2\pi)$ of the coupled particle-field system 
\begin{eqnarray*}
 \nabla_t^2 q(t) &=& -\nabla(\varphi*\rho)(t,q(t)) - \nabla (\varphi^{\text{ext}}*\rho)(t,q(t)),\\
 \partial_t^2 \varphi(t,x) &=& \Delta\varphi(t,x)-\varphi(t,x)-\rho(x-q(t)),
\end{eqnarray*}
where we further allow for an exterior potential which is given a priori and is again time-periodic with period $T$ and space-periodic with period $2\pi$, $\varphi^{\text{ext}}(t+T,x)=\varphi^{\text{ext}}(t,x)=\varphi^{\text{ext}}(t,x+2\pi)$. Note that this set of equations can be viewed as a simplified version of the coupled Maxwell-Lorentz equations for the electric field $\mathbf{E}(t)$ and the magnetic field $\mathbf{B}(t)$ on $\T^3=\R^3/(2\pi\Z)^3$ and a charged particle with normalized charge whose locus $q(t)$ is constrained to the submanifold $Q\subset\T^3$, see \cite{Spo},
\begin{eqnarray*}
&&\nabla_t^2 q(t)=\pi_{q(t)}\left((\mathbf{E}*\rho)(t,q(t))+\partial_t q(t)\times(\mathbf{B}*\rho)(t,q(t))\right)\\
&&\left(\nabla\times\mathbf{E}\right)(t,x)=-\partial_t\mathbf{B}(t,x),\quad\left(\nabla\cdot\mathbf{E}\right)(t,x)=-\rho(q(t)-x),\\ &&\left(\nabla\times\mathbf{B}\right)(t,x)=\partial_t q(t)\cdot\rho(q(t)-x)+\partial_t\mathbf{E}(t,x),\quad\left(\nabla\cdot\mathbf{B}\right)(t,x)=0,
\end{eqnarray*}
where for the first equation $\pi_q:\R^3=T_q\T^3\to T_q Q$ denotes the orthogonal projection onto the tangent space of $Q$ at every $q\in Q$. In particular, as discussed in \cite[Section 2.3]{Spo}, \emph{it is important not just from the viewpoint of mathematics but also from the viewpoint of physics to consider extended particles rather than point particles, that is, to work with a sufficiently regular density function $\rho$ rather than the distribution $\delta$.} It can be shown that the coupled particle-field system is equivalent to the system 
\begin{eqnarray*}
 \partial_t q = p,&& \nabla_t p= -\nabla (\varphi*\rho)(q)-\nabla(\varphi^{\text{ext}}*\rho)(q)\\
 \partial_t \varphi = B\pi,&& \partial_t \pi = - B\varphi-B^{-1}\rho(q-\cdot).
\end{eqnarray*}
with $B=\sqrt{1-\Delta}$. It can be seen that this modified system is still Hamiltonian with Hamiltonian function given by
\[H_t(q,p,\varphi,\pi)=H_{\text{field}}(\varphi,\pi)+H_{\text{inter},t}(q,\varphi,\pi)+H_{\text{part},t}(q,p)\]
with
\begin{eqnarray*}
H_{\text{field}}(\varphi,\pi)&=&\frac{1}{2}\langle\varphi,B\varphi\rangle_{\Hil}+\frac{1}{2}\langle\pi,B\pi\rangle_{\Hil},\\ H_{\text{inter}}(q,\varphi,\pi)&=&(\varphi*\rho)(q)=\langle\varphi,B^{-1}\rho(q-\cdot)\rangle_{\Hil},\\H_{\text{part},t}(q,p)&=&\frac{1}{2}|p|^2+(\varphi^{\text{ext}}_t*\rho)(q),
\end{eqnarray*}
defined on the extended phase space $T^*Q\times H^{\frac{1}{2}}(\T^N,\C)$, where the symplectic form $\omega_{\Hil}$ on $\Hil=H^{\frac{1}{2}}(\T^N,\C)$ is given by the standard complex structure $J_{\Hil}\cdot(\varphi,\pi)=(-\pi,\varphi)$ and the standard real inner product $\langle\cdot,\cdot\rangle_{\Hil}$ on $H^{\frac{1}{2}}(\T^N,\C)$ given by $\displaystyle\langle f,g\rangle_{\Hil}=\int_{\T^N} f(x) (Bg)(x)\,dx$, see \cite{Kuk1}, \cite{Kuk2} for the Hamiltonian structure for the wave equation. 
\par
In this paper we continue with our project of defining a Floer homology theory for Hamiltonian PDEs with regularizing nonlinearities as arising e.g. in the study of Hamiltonian particle-field systems. The latter are obtained by coupling a classical mechanical system with a (linear) classical field theory; in particular, the underlying phase space is the direct product $M\times\Hil$ of a finite-dimensional symplectic manifold $M$ being the phase space for the mechanical system with an infinite-dimensional linear symplectic Hilbert space $\Hil$ being the phase space for the field theory. \par
While in our papers \cite{FL}, \cite{FL2} we could already prove an infinite-dimensional analogue of Gromov-Floer compactness for Hamiltonian PDE with regularizing nonlinearities, building on our earlier work in \cite{F}, and use it to prove cuplength estimates for the number of $T$-periodic orbits, in this paper we establish the second main ingredient for studying moduli spaces of Floer curves, namely the nonlinear Fredholm property for the nonlinear Cauchy-Riemann operator with zeroth order Hamiltonian term. As for the compactness theorems, apart from the obvious new problems arising from the infinite-dimensionality of the target space, the main new arising challenge is a small divisor problem: while for generic time periods the underlying linear Hamiltonian PDE only has the trivial periodic solution and the return map only has eigenvalues different from one, there is always a subsequence of
eigenvalues converging to one. On the other hand, in finite-dimensional Floer theory it is crucial to assume that the eigenvalues of the linearized return map stay away from one. We show that this problem can be resolved by considering suitable new Hilbert space completions and introducing the notion of nondegeneracy up to small divisors.\\\par

This paper is organized as follows: In section $2$ we introduce symplectic Hilbert spaces and scales, and specify the class of admissible Hamiltonians which we consider in this paper. Among other concrete examples we discuss Hamiltonian particle-systems obtained by coupling standard examples of linear Hamiltonian field theories with Hamiltonian mechanical systems on the cotangent bundle of a submanifold. In section $3$ we formulate two conjectures about Betti number estimates for the minimal number of $T$-periodic orbits that generalize both, the corresponding famous Betti number estimates for closed symplectic manifolds from \cite{Fl} and for cotangent bundles of closed manifolds from \cite{AS}, as well as the corresponding cuplength estimates proven in \cite{FL} and \cite{FL2} for the infinite-dimensional Hamiltonian systems that we consider. Furthermore we outline the definition of a Floer homology for the corresponding class of admissible Hamiltonian functions, including the definition of moduli spaces of Floer curves. In section $4$ we finally work out a nonlinear Fredholm theory that it is designed in order to prove that moduli spaces of Floer curves are smooth finite-dimensional manifolds under genericity assumptions. After discussing the new problems that arise due to small divisors, we in particular show how they can get resolved naturally by introducing suitable modified Hilbert space completions, which in turn leads to new notion of nondegeneracy for $T$-periodic orbits, called nondegeneracy up to small divisors. In the final section $5$ we prove that the latter property is generic for the class of Hamiltonian systems with regularizing nonlinearities that we consider.    

\section{Hamiltonian PDEs with regularizing nonlinearities}\label{HamPDE-section}
We start with recalling the mathematical framework for Hamiltonian PDEs from \cite{Kuk1}, \cite{Kuk2}, see also \cite{AM}. Let $\Hil$ denote a separable Hilbert space which is equipped with a bilinear two-form $\omega:\Hil\times\Hil\to\R$ which is (strongly) symplectic in the sense that it is anti-symmetric and $\omega=\omega_\Hil:\Hil\to\Hil^*$ is a linear isomorphism. Analogous to the finite-dimensional case, there exists a complete Darboux basis $(e_n^\pm)_{n\in\Z^N}$ of $\Hil$ in the sense that $\omega(e_n^+,e_m^-)=\delta_{n,m}$. Defining a linear operator $J=J_\Hil:\Hil\to\Hil$ by $Je_n^\pm=\mp e_n^\mp$, we find that $J$ is a compatible complex structure in the sense that $J^2=-1$ and $\omega=\langle J\cdot,\cdot\rangle$ for an equivalent real inner product $\langle \cdot,\cdot\rangle$ with corresponding norm $|\cdot|$ on $\Hil$ which we fix from now on. Note that this allows us to identify $\Hil$ with a complex subspace of $\Hil\otimes_\R\C$ with complex basis $z_n=(e_n^++ie_n^-)/\sqrt{2}$ and hence also with the space $\ell^2(\Z^N,\C)$. After choosing for every $n\in\Z^N$ a positive number $\theta_n$ such that $\theta_n=O(|n|)$, there is a natural sequence of Hilbert spaces $\Hil_h$, $h\in\R$ by requiring that the $\Hil_h$-norm of $z_n$ is equal to $\theta_n^h$. It is immediate to check that $(\Hil_h)_{h\in\R}$ is a symplectic Hilbert scale on $\Hil$ in the sense that $\Hil_0=\Hil$, $\Hil_h$ is a subset of $\Hil_{h'}$ and the embedding is compact and dense for all $h>h'$, and the symplectic form defines a linear isomorphism $\omega:\Hil_h\stackrel{\cong}{\to}\Hil_{-h}^*$ for all $h\in\R$. We define $\Hil_\infty=\bigcap_{h\in\Z}\Hil_h$ and $\Hil_{-\infty}=\bigcup_{h\in\Z}\Hil_h$; since $\omega:\Hil_\infty\to\Hil_\infty^*$ is only injective, we stress that $\omega$ only defines a weakly symplectic form on the Frechet space $\Hil_\infty$.\par
We generalize this setup to the case where $\wt{M}=M\times\Hil$ is the product of a finite-dimensional symplectic manifold $M=(M,\omega_M)$ with a (strongly) symplectic Hilbert space $\Hil$. $\wt{M}$ carries product symplectic form $\omega=\pi_M^*\omega_M+\pi_{\Hil}^*\omega_{\Hil}$, where $\pi_M,\pi_\Hil$ denotes the projection onto $M$, $\Hil$, respectively, and a scale structure given by $\wt{M}_h:=M\times\Hil_h$ for all $h\in\R$. In particular, note that $\wt{M}_\infty=M\times\Hil_\infty$ is a weakly symplectic infinite-dimensional Frechet manifold. Furthermore, choosing an $\omega_M$-compatible almost complex structure $J_M$ on $M$ in the sense that $\omega_M(\cdot,J_M\cdot)$ defines a Riemannian metric on $M$, we find that $J=\Diag(J_M,J_{\Hil})$ defines an $\omega$-compatible almost complex structure on $\wt{M}=M\times\Hil$.\par

As in \cite{FL} in this work we consider $T$-periodic Hamiltonians $H_t:\wt{M}\to\R$, $H_t=H_{t+T}$, of the form 
\[
H_t(u)=H^A(u)+F_t(u):=\frac{1}{2}\langle A\pi_\Hil(u),\pi_\Hil(u)\rangle+F_t(u).
\]
Here $A$ is a self-adjoint differential operator of order $d>0$ and we assume from now that the complete Darboux basis $(e_n^\pm)_{n\in\Z^N}$ of $\Hil$ consists of eigenfunctions of the differential operator $A$ with corresponding real eigenvalues $\lambda_n\in\R$, $n\in\Z^N$. Furthermore $F_t=F_{t+T}$ is a $T$-periodic Hamiltonian function modelling the regularizing nonlinearity, see \Cref{admissible-2} below. 
\begin{definition}
\label{admissible-1}
We call the pair $(A,T)$ \emph{admissible} when the following condition is satisfied: For $\epsilon_n\in(-\pi/T,\pi/T]$ being defined by $\epsilon_n:=\lambda_n\mod2\pi/T$ there exists $h_0\geq 2d$ such that for every $h>h_0$ there exists $c=c_h>0$ such that $|\epsilon_n|>c \theta_n^{-h}$.
\end{definition} 
An irrational number $\sigma$ is called Diophantine if there exists $c>0$ and $r>0$ such that $$\inf_{p\in\Z}\Big|\sigma-\frac{p}{q}\Big|\,\geq\, c\cdot q^{-r}\,\,\textrm{for all}\,\, q\in\N,$$ and the irrationality measure $r_0\geq 2$ of $\sigma$ is defined as the infimum of the set of $r>0$; note that for generic $\sigma$ we have $r_0=2$. \par
Before we start to motivate our definition, we first give two examples.  \begin{example}[Wave equation]
\label{wave}
Recall that the second-order wave equation $\partial_t^2\varphi=\Delta\varphi-a\cdot\varphi$ with $a\in\N\cup\{0\}$ can be rewritten as a first-order system 
\[\partial_t\varphi=B\pi,\,\,\partial_t=-B\varphi\,\,\text{with}\,\,B=\sqrt{a-\Delta}.\] Setting $u_{\Hil}=(\varphi,\pi)$ we see that this system is Hamiltonian for the densly defined Hamiltonian \[H_t(u_{\Hil})=H^A(u_{\Hil})=\frac{1}{2}\langle u_{\Hil},Au_{\Hil}\rangle_{\Hil}=\frac{1}{2} \langle\varphi,B\varphi\rangle_{\Hil}+\frac{1}{2}\langle\pi,B\pi\rangle_{\Hil}\,\,\text{with}\,\,A=\Diag(B,B)\] on the separable Hilbert space $\Hil=H^{\frac{1}{2}}(\T^N,\C)$, equipped with the symplectic form $\omega_{\Hil}=\langle J_{\Hil}\cdot,\cdot\rangle_{\Hil}$ given by the complex structure $J_{\Hil}$ and the inner product on $H^{\frac{1}{2}}(\T^N,\C)$ from \Cref{motivation-section}, where for $a=0$ we restrict ourselves to the Hilbert subspace $H^{\frac{1}{2}}_0(\T^N,\C)$ of functions with zero mean. There exists a complete Darboux basis $(e_n^\pm)_{n\in\Z^N}$ of $\Hil$ of the form $e_n^+=(\xi_n,0)$, $e_n^-=(0,\xi_n)$, where $(\xi_n)_{n\in\Z^N}$ is a suitably normalized complete basis of $H^{\frac{1}{2}}(\T^N,\R)$ in terms of sine and cosine functions, which furthermore consists of eigenfunctions of the linear operator $A$ with corresponding real eigenvalues $\lambda_n=\sqrt{n^2+a}$ with $n^2=n_1^2+\ldots+n_N^2$ for every $n=(n_1,\ldots,n_N)\in\Z^N$. Assume that $T^2/(2\pi)^2$ is a Diophantine irrational number with irrationality measure $r_0\geq 2$. Then we know that for every $r>r_0$ there exists $c>0$ such that \[\inf_{m\in\N}\left|\left(\frac{T}{2\pi}\right)^2-\frac{m^2}{n^2+a}\right|\geq c\cdot (n^2+a)^{-r};\] note that here and below we assume that $n\neq 0$ if $a=0$. Since \[\left(\frac{T}{2\pi}\right)^2-\frac{m^2}{n^2+a}=\left(\frac{T}{2\pi}-\frac{m}{\sqrt{n^2+a}}\right)\cdot\left(\frac{T}{2\pi}+\frac{m}{\sqrt{n^2+a}}\right),\] and the second factor is approximately equal to $2\cdot T/(2\pi)$ whenever the first factor is close to zero, it follows that there also exists some $c'>0$ and $r'=r-\frac{1}{2}$ such that for \[|\epsilon_n|=\frac{1}{T}\inf_{m\in\Z}\left|T\cdot\sqrt{n^2+a}-m\cdot 2\pi\right|=\frac{2\pi}{T}\cdot\sqrt{n^2+a}\cdot\inf_{m\in\Z}\left|\frac{T}{2\pi}-\frac{m}{\sqrt{n^2+a}}\right|\] we find that \[|\epsilon_n|> c'\cdot (n^2+a)^{-r'}.\] In particular, for $h>h_0=2r_0-1\geq 2$ there exists $c''>0$ such that $|\epsilon_n|>c''\theta_n^{-h}$.
\end{example}
\begin{example}[Electromagnetism]\label{EM}
Choosing a vector potential $\mathbf{A}$ satisfying the Coulomb gauge condition for the magnetic field $\mathbf{B}$ and the electric field $\mathbf{E}$ in the sense that \[\mathbf{B}=\nabla\times\mathbf{A},\quad\mathbf{E}=\partial_t\mathbf{A},\quad\text{and}\quad\nabla\cdot\mathbf{A}=0,\] we find that the charge-free Maxwell equations are equivalent to the fact that the vector potential $\mathbf{A}$ satisfies the wave equation, \[\partial_t^2\mathbf{A}=\Delta\mathbf{A}.\] Imposing again the condition that $\mathbf{A}$ is $2\pi$-periodic with respect to the space coordinates, it follows that the linear differential operator $A$ is as in \Cref{wave}, but the corresponding Hilbert space $\Hil$ is now the Hilbert space completion of the subspace of pairs of divergence-free vector fields on $\T^3=\R^3/(2\pi\Z)^3$ with zero mean, $\Hil=H^{\frac{1}{2}}_{0,*}(\T^3,\C^3)\subset H^{\frac{1}{2}}_0(\T^3,\C^3)$, see \cite{BaGa}. While it directly follows from \Cref{wave} that still $(A,T)$ is admissible whenever $T^2/(2\pi)^2$ is a Diophantine irrational number, since $a=0$ it can be readily checked that it is sufficient to assume that $T/(2\pi)$ is a Diophantine irrational number. 
\end{example}
\begin{example}[Schrödinger equation]
\label{Schrodinger}
The Schrödinger equation \[i\partial_t u_{\Hil}=-\Delta u_{\Hil}\] is a Hamiltonian PDE for the densly-defined Hamiltonian function \[H_t(u_{\Hil})=H^A(u_{\Hil})=\frac{1}{2}\langle u_{\Hil},Au_{\Hil}\rangle=-\frac{1}{2}\int_0^X |\nabla u_{\Hil}(x)|^2\,dx\,\,\text{with}\,\,A=\Delta\] on the separable Hilbert space $\Hil=L^2_0(\T^N,\C)$ of square-integrable functions with zero mean. It has a complete Darboux basis given by $e_n^+=e_n$, $e_n^-=-ie_n$, where $(e_n)_{n\in\Z^N\backslash\{0\}}$ is the complete system of eigenfunctions of $\Delta$ given by $e_n(x)=e^{inx}/\sqrt{2\pi}$ with eigenvalues $\lambda_n=n^2\neq 0$. Assume that $T/(2\pi)$ is a Diophantine irrational number with irrationality measure $r_0\geq 2$. Then we know that for every $r>r_0$ there exists $c>0$ such that \[|\epsilon_n|=\inf_{m\in\Z}\left|n^2-m\cdot\frac{2\pi}{T}\right|=\frac{2\pi}{T}\cdot n^2\cdot\inf_{m\in\Z} \left|\frac{T}{2\pi}-\frac{m}{n^2}\right|>c\cdot (n^2)^{1-r}.\] In particular, for every $h>h_0=\max\{r_0-1,2d\}$ with $d=2$ there exists $c'>0$ such that $|\epsilon_n|>c'\theta_n^{-h}$.  
\end{example}

Since the differential operator $A$ is of positive order $d>0$ and the Hamiltonian $H^A$ is hence only densely defined on $\wt{M}_{\floor{d/2}}\subset\wt{M}$, we rather work with its flow $\phi_t=\phi^A_t$ defined by \[\partial_t\phi^A_t = X^A\circ\phi^A_t = J_{\Hil} A\circ\phi^A_t\] which acts trivially on $M$ and is unitary and defined on all of $\Hil$. Indeed, working from now on in the complexified Hilbert space $\Hil\otimes_\R\C$ with complex basis $z_n=(e_n^++ie_n^-)/\sqrt{2}$, the flow $\phi^A_t$ of $H^A$ acts on an element $u\in \Hil$, written as $u=\sum_n\hat{u}(n)z_n$ with complex coefficients $\hat{u}(n)$, as
\[
\phi^A_tu=\sum_n\hat{u}(n)e^{i\lambda_nt}z_n=\sum_n\hat{u}(n)e^{i\epsilon_nt}z_n.
\]
For every $t\in\R$ note that $\phi^A_t$ preserves the Hilbert scale, $\phi^A_t:\Hil_h\to\Hil_h$, for every $h\in\R$. Furthermore it directly follows from the admissibility condition that $\phi^A_T$ has no eigenvalue equal to one. On the other hand, as it can be seen from our two examples, there typically exists a subsequence of the sequence of eigenvalues $e^{i\epsilon_n T}$ of $\phi^A_T$ which converges to one. This is known as \emph{small divisor problem}. In order to deal with the resulting problems, in our compactness papers \cite{FL}, \cite{FL2} we require that the Hamiltonian $F_t$ defining the nonlinearity can be approximated better by finite-dimensional Hamiltonians than the eigenvalues of $\phi^A_T$ approach one. Following \cite[Definition 3.1]{FL} we define
\begin{definition}\label{admissible-2}
We call the $T$-periodic nonlinearity Hamiltonian $F_t:\wt{M}=M\times\Hil\to\R$ \emph{admissible} if it is $h$-regularizing for some $h>h_0\geq 2d$ in the sense that it defines a $(\floor{h/d}+2)$-times continuously differentiable map from $\wt{M}_{-h}=M\times\Hil_{-h}$ to $\R$ which is $\floor{h/d}$-times continuously differentiable with respect to $t$, where $d>0$ denotes the order of the differential operator $A$. 
\end{definition}
Note that the fact that $F_t$ extends to a map from $\wt{M}_{-h}=M\times\Hil_{-h}$ to $\R$ implies that the symplectic gradient $X^F_t(u)=(J\nabla F_t)(u)$ is an element of $T_{u_M}M\oplus\Hil_{+h}\subset T_{u_M}M\oplus\Hil=T_u\wt{M}$ for every $u=(u_M,u_{\Hil})\in M\times\Hil\subset M\times\Hil_{-h}=\wt{M}_{-h}$. On the other hand, setting $G_t:=F_t\circ\phi^A_{-t}$, recall that $T$-periodic solutions \begin{align}
\label{hameqn}
\partial_tu=X^H_t(u);\qquad u(t+T)=u(t)
\end{align}
in $\P(H)$ to the Hamiltonian equation with Hamiltonian $H_t$ are in univocal correspondence with $\phi^A_T$-periodic solutions 
\begin{align}
\label{ghameqn}
\partial_tu=X^G_t(u);\qquad u(t+T)=\phi^A_{-T}u(t)
\end{align}
in $\P(\phi^A_T,G)$ to the Hamiltonian equation with Hamiltonian $G_t$. While working with the time-$T$ map $\phi^A_T$ instead of with the Hamiltonian $H^A$ itself avoids the problems that arise from the fact that $H^A$ is only densly defined, it can be seen in \cite[Lemma 6.1]{FL} that we now need to accept that, even if we assumed that $F_t$ depends smoothly on $t$, $G_t:\Hil\to\R$ is only $\floor{h/d}$-times continuously differentiable with respect to $t\in\R$: Indeed, in order to compute the $t$-derivatives of $G_t=F_t\circ\phi^A_{-t}$, we also need to take into account the $t$-derivatives of $t\mapsto\phi^A_{-t}$ which are given by $\partial_t^{\alpha}\phi^A_{-t}=(-JA)^{\alpha}\circ\phi^A_{-t}:\Hil\to\Hil_{-\alpha d}$ with $\Hil_{-\alpha d}\subset\Hil_{-h}$ as long as $\alpha\leq\floor{h/d}$, where $d>0$ is the order of $A$ and we use that $\phi^A_t$ preserves the symplectic Hilbert scale. The use of the different orders of differentiability in \Cref{admissible-2} are motivated by the following proposition, see also \cite[Lemma 6.1]{FL}.
\begin{proposition}\label{gradient+Hessian}
Assume that $F_t$ is $h$-regularizing in the sense of \Cref{admissible-2}. Then the gradient $\nabla G(t,u)=\nabla G_t(u)$ and the Hessian $\nabla^2 G(t,u)=\nabla^2 G_t(u)$ of $G_t=F_t\circ\phi^A_{-t}$ define $\floor{h/d}(\geq 2)$-times continuously differentiable maps $\nabla G:\R\times\Hil\to\Hil_{+h}$ and $\nabla^2 G:\R\times\Hil\to\Hil_{-h}^*\otimes\Hil_{+h}$.
\end{proposition}
While \Cref{admissible-2} just seems to be suited in order to generalize Floer theory to infinite dimensions, we show below that the coupled particle-field system as well as the coupled Maxwell-Lorentz equations discussed above provide us with examples. Moreover, still in the case when $M=T^*Q$, we introduce a much larger but still natural class of Hamiltonian systems describing Hamiltonian mechanical systems coupled with Hamiltonian field theories. But before, in order to see that one can readily write down examples also in the case of general finite-dimensional symplectic manifolds, we first consider the following generalization of the examples given in \cite[Example 1.1+1.2]{FL}.
\begin{proposition}
Consider $\Hil$ and $A$ as in \Cref{wave}, \Cref{EM}, or as in \Cref{Schrodinger}. Defining \[F_t(u)=F_t(u_M,u_{\Hil})=\int_{\T^N} f_t(u_M,(u_{\Hil}*\rho)(x),x)\,dx\] with $f_t=f_{t+T}:M\times\R^2\times\T^N\to\R$ ($f_t=f_{t+T}:M\times\R^6\times\T^3\to\R$ in \Cref{EM}) being $(\floor{h/d}+2)$-times continuously differentiable and  $\floor{h/d}$-times continuously differentiable with respect to $t\in\R$, we find that $F_t$ is admissible in the sense of \Cref{admissible-2}, provided that $\rho:\T^N\to\R$ is $h$-times continuously differentiable for some $h>h_0$.
\end{proposition}
In particular, when the shape function $\rho$ of the particle is smooth and $f_t$ is smooth and smoothly depending on $t$, then $F_t$ is $h$-regularizing for \emph{all} $h>h_0$.
\begin{proof}
When $u_{\Hil}\in\Hil_{-h}$ and $\rho\in C^h$, then $u_{\Hil}*\rho\in C^0$ for all examples. For $\alpha\leq \floor{h/d}+2$ it follows that the maps $x\mapsto\partial_2^{\alpha}f_t(u_M,(u_{\Hil}*\rho)(x),x)$ obtained by differentiating $f_t$ with respect to the second argument are continuous and hence (square-) integrable over $\T^N$. Altogether this is sufficient to prove that $F_t$ is a $(\floor{h/d}+2)$-times continuously differentiable map from $\wt{M}_{-h}=M\times\Hil_{-h}$ to $\R$ in both examples which is $\floor{h/d}$-times continuously differentiable with respect to $t\in\R$. 
\end{proof}
But now let us turn to the most important class of examples, at least from the viewpoint of classical physics. Generalizing the example of the Hamiltonian system coupling a Hamiltonian mechanical system with the wave equation from above and motivated by \cite{BaGa}, \cite{Kun}, \cite{Spo}, we consider the following natural class of $h$-regularizing Hamiltonians. 
\begin{proposition}\label{particle-field} Consider $\Hil$ and $A$ as in \Cref{wave}, \Cref{EM}, or as in \Cref{Schrodinger}, and let $Q\subset\T^N$ be a smooth submanifold ($N=3$ for \Cref{EM}). Let $F_t$ be a $T$-dependent Hamiltonian function on $M\times\Hil$ of the form
\[F_t(u)=F_t(q,p,u_{\Hil})=f_t(q,p,(u_{\Hil}*\rho)(q)),\]
where $f_t$ is a $\floor{h/d}+2$-times continuously differentiable map from $T^*Q\times\R^2$ ($T^*Q\times\R^6$ in \Cref{EM}) to $\R$ being $\floor{h/d}$-times continuously differentiable with respect to $t\in\R$ and we additionally assume that the density function $\rho:\T^N\to\R$ of the particle is $(h+\floor{h/d}+2)$-times continuously differentiable for some $h>h_0$. Then $F_t$ is $h$-regularizing in the sense of \Cref{admissible-2}. 
\end{proposition}
Again, when the shape function $\rho$ of the particle is smooth and $f_t$ is smooth and smoothly depending on $t$, then $F_t$ is $h$-regularizing for \emph{all} $h>h_0$. While the regularity of the density function $\rho$ is hence clearly needed in order to have a sufficiently regularizing nonlinearity, we would like to stress again that the use of the regularizing density function $\rho$ instead of the $\delta$-distribution is well-motivated from the viewpoint of physics, see e.g. \cite[Section 2.3]{Spo} for a detailed discussion.\par
\begin{proof}
It suffices to analyze the map \[Q\times\Hil_{-h}\to\R,\,\,(q,u_{\Hil})\mapsto (u_{\Hil}*\rho)(q)=\langle u_{\Hil},(B^{-1})\rho(q-\cdot)\rangle.\] Since $\rho$ is $(h+\floor{h/d}+2)$-times continuously differentiable, in all three cases it follows for every $u\in\Hil_{-h}$ that the function $u_{\Hil}*\rho$ is $(\floor{h/d}+2)$-times continuously differentiable. Since the derivative with respect to $\Hil$ is given by $(B^{-1})\rho(q-\cdot)$ and hence is independent of $u_{\Hil}\in\Hil_{-h}$, it follows that also mixed derivatives exist and are continuous up to order $\floor{h/d}+2$. 
\end{proof}
\begin{example}[coupled particle-field system] Let $\Hil$ and $A$ be as in \Cref{wave}. Defining $f_t:T^*Q\times(\R\times\R)\to\R$ by \[f_t(q,p,r_1,r_2)=\frac{1}{2}|p|^2+(\varphi^{\text{ext}}_t*\rho)(q)+r_1\] we find that $F_t(q,p,u_{\Hil})=f_t(q,p,(u_{\Hil}*\rho)(q))=H_{\text{part},t}(q,p)+H_{\text{inter}}(q,u_{\Hil})$. 
\end{example}
\begin{example}[coupled Maxwell-Lorentz equations]
Let $\Hil$ and $A$ be as in \Cref{EM}. It can be shown that the coupled Maxwell-Lorentz equations are obtained by setting  \[F_t(q,p,u_{\Hil})=f_t(q,p,(u_{\Hil}*\rho)(q))=\frac{1}{2}\left|p-\pi_q\left((\mathbf{A}*\rho)(q)\right)\right|^2\] with the map \[f_t:T^*Q\times(\R^3\times\R^3)\to\R\quad\text{given by}\quad f_t(q,p,r_1,r_2)=\frac{1}{2}\left|p-\pi_q r_1\right|^2,\]  where we use the inner product on $\R^3$ to define for every $q\in Q\subset\T^3$ the orthogonal projection $\pi_q: \R^3=T_q\T^3\to T_q Q$ and to identify $T_q^*Q\cong T_q Q$. 
\end{example}

\section{Moduli spaces of Floer curves}
\label{moduli-section}
Building on previous work \cite{FL} and \cite{FL2} which itself builds on \cite{F}, it is the goal to construct a full Floer homology theory for time-periodic Hamiltonians $H_t=H^A+F_t$ on $\wt{M}=M\times\Hil$ which are admissible in the sense of \Cref{admissible-1} and \Cref{admissible-2}. More concretely, we want to focus on the case where the finite-dimensional symplectic manifold $M$ is either closed with $\pi_2(M)=\{0\}$ or the cotangent bundle $T^*Q$ of a closed manifold $Q$. We follow the by now standard approach to constructing Floer homology which can be found in \cite{Sa} and \cite{AD}. The main application of the results of this paper is to establish Betti number estimates for the minimal number of $T$-periodic orbits, upgrading our already existing results in terms of cuplenghts in \cite{FL} and \cite{FL2}.

\begin{conjecture}
\label{arnoldconjecture}
Assume that finite-dimensional symplectic manifold $(M,\omega_M)$ is closed and aspherical. For every Hamiltonian $H_t=H^A+F_t$ on $M\times\Hil$ which is admissible in the sense of \Cref{admissible-1} and \Cref{admissible-2} with generic choice of $F_t$ the number of solutions to the Hamiltonian equation \eqref{ghameqn} and hence to \eqref{hameqn} is bounded from below by the sum of the $\Z_2$-Betti numbers of $M$. \end{conjecture}\par

For the case when the finite-dimensional symplectic manifold is the cotangent bundle of a closed manifold $Q$, let us restrict the general class of Hamiltonians to the case of particle-field Hamiltonians on $T^*Q\times H^{\frac{1}{2}}(T^N,\C)$ for which we could already establish the necessary $C^0$-bounds in \cite{FL2}. This means that \begin{eqnarray*}
H^A(u)\quad\text{as in \Cref{wave} and}\quad F_t(u)=F_{\prt,t}(q,p)+F_{\inter,t}(q,u_{\Hil}),
\end{eqnarray*}
where the particle Hamiltonian $F_{\prt,t}\in C^{\infty}(T^*Q,\R)$ is asymptotically quadratic with respect to the momentum coordinates $p$ in the sense that 
\begin{itemize}
\item[(F1)] $dF_{\prt,t}(q,p)\cdot p\frac{\partial}{\partial p} - F_{\prt,t}(q,p)\geq c_0 |p|^2 - c_1$, for some constants $c_0>0$ and $c_1\geq 0$,
\item[(F2)] $|\nabla_q F_{\prt,t}(q,p)|\leq c_2 (1+|p|^2),\,\,|\nabla_p F_{\prt,t}(q,p)|\leq c_2 (1+|p|)$ for some constant $c_2\geq 0$,
\end{itemize} 
while for the interaction Hamiltonian we require that \begin{itemize}
\item[(F3)] $F_{\inter,t}(q,u_{\Hil})=f_t((u_{\Hil}*\rho)(q))$, where $f_t$ has bounded first derivatives.
\end{itemize}
In our upcoming paper \cite{FL3} we prove the following
\begin{conjecture}
Assume that the finite-dimensional symplectic manifold is the cotangent bundle of a closed manifold $Q$ and that the  particle-field Hamiltonian $H_t=H^A+F_t$ on $T^*Q\times H^{\frac{1}{2}}(T^N,\C)$ is of the form $F_t=F_{\text{part},t}+F_{\text{inter},t}$ satisfying (F1), (F2), (F3) stated above. Then the number of solutions to the Hamiltonian equation \eqref{ghameqn} and hence to \eqref{hameqn} is bounded from below by the sum of the $\Z_2$-Betti numbers of the space of contractible loops in $Q$, provided that for every $T$-periodic orbit the linearized return map has no eigenvalue equal to one.
\end{conjecture}\par

Denote by $\P=\P(\phi,G)$ the set of $\phi^A_T$-periodic orbits $u$ of $G_t$, with $\phi=\phi^A_T$. As in finite dimensions it is the goal to define Floer homology as the homology of a chain complex $$HF_*(A,F)=HF_*(\phi,G)=H_*(CF_*,\partial),$$ where the underlying chain groups $CF_*$ are defined as $\Z_2$-vector spaces spanned by the $T$-periodic orbits $u$ with the same Conley-Zehnder index $\mu_{CZ}(u)$, $CF_*=\Z_2\langle u\in\P: \mu_{CZ}(u)=*\rangle$; note that away from the example of cotangent bundles we assume here that $\pi_2(M)=\{0\}$ such that the Conley-Zehnder index can be defined without extra choices. The boundary operator $\partial:CF_*\to CF_{*-1}$ is to be defined as $$\partial u^+=\sum_{u^-}\#\M(u^+,u^-)/\R\cdot u^-$$ by counting in $\Z_2$ elements in the quotient set $\M(u^+,u^-)/\R$. \par

Here $\M=\M(u^+,u^-)=\M(u^+,u^-;G)$ denotes the moduli space of Floer curves $\wt{u}$ connecting $u^+$ and $u^-$, that is, the set of all $\wt{u}:\R^2\to\wt{M}$ satisfying the periodicity condition $\wt{u}(s,t+1)=\phi^{-1}(\wt{u}(s,t))$, the asymptotic condition $\wt{u}(s,t)\to u^\pm(t)$ as $s\to\pm\infty$, and the Floer equation $$\CR_G\wt{u}=\partial_s\wt{u}+J(\wt{u})\partial_t\wt{u}+\nabla G_t(\wt{u})=0.$$ Note that $\M(u^+,u^-)$ carries a natural $\R$-action and for the definition of the boundary operator $\partial:CF_*\to CF_{*-1}$ to make sense we first wish to prove that $\M(u^+,u^-)$ is a one-dimensional manifold for generic choices of $H_t=H^A+F_t$, that is, of $F_t$. \par

For this we use that each moduli space $\M(u^+,u^-;G)$ can be described as the zero set of the perturbed Cauchy-Riemann operator $\CR_G$, also called nonlinear Floer operator, $$\M(u^+,u^-;G)=\CR_G^{-1}(0)\subset\B(u^+,u^-),$$ where $\B=\B(u^+,u^-)$ is a suitable Hilbert space completion defined below in \Cref{Fredholmsection} of the space $C^{\infty}_{\phi}(u^+,u^-)$ of smooth maps $\wt{u}:\R^2\to\wt{M}=M\times\Hil$ satisfying the periodicity condition $\wt{u}(s,t+1)=\phi^{-1}(\wt{u}(s,t))$ and the asymptotic condition $\wt{u}(s,t)\to u^\pm(t)$ as $s\to\pm\infty$. As in finite-dimensional Floer theory the nonlinear Floer operator can be viewed as a section in a Hilbert space bundle $\E=\E(u^+,u^-)\to\B(u^+,u^-)$. Assuming that the nonlinear Floer operator $\CR_G$ is transversal to the zero section and nonlinear Fredholm in the sense that its linearization $D_{\wt{u}}: T_{\wt{u}}\B\to \E_{\wt{u}}$ at every 
$\wt{u}\in\M=\CR_G^{-1}(0)$ is linear Fredholm, it follows from the implicit function theorem for nonlinear Fredholm maps between Hilbert manifolds that $\M$ is a smooth manifold of finite dimension given by the Fredholm index and the tangent space $T_{\wt{u}}\M$ to the moduli space $\M=\M(u^+,u^-)$ is given by the kernel of the linearization $D_{\wt{u}}$ at every $\wt{u}\in\M$, $$T_{\wt{u}}\M=\ker D_{\wt{u}}\subset T_{\wt{u}}\B.$$  \par

Apart from the more obvious problems that appear in Fredholm theory when passing from finite to infinite dimensional target spaces, the main difficulties arise from the presence of small divisors. Indeed, while in finite-dimensional Floer theory we always need to assume that the Hamiltonian is chosen generically such that the linearized return map along all periodic orbits has no eigenvalue equal to one, in our case we always need to assume that there is a subsequence of the infinite sequence of eigenvalues which converges to one. After introducing the notion of nondegeneracy up to small divisors, we show that the Fredholm property of the linearized Floer operator can still be established for generic choices when working with suitable but non-canonical Sobolev completions. 

\section{Fredholm property}
\label{Fredholmsection}
We start with the following observation: Since the map $\phi=\phi^A_T:\Hil\to\Hil$ has no eigenvalues equal to one, it follows as in finite dimensions that the linear map $i\partial_t: C^{\infty}_\phi(\R,\Hil)\to C^{\infty}_\phi(\R,\Hil)$ has no kernel, where $C^{\infty}_\phi(\R,\Hil)$ consists of maps $\xi:\R\to\Hil$ with $\xi(t+T)=\phi^{-1}\cdot\xi(t)$, where we set $\phi:=\phi^A_T$. Fix some natural number $k>1$. Defining $H^k_\phi(\R,\Hil)$ to be the completion of $C^{\infty}_\phi(\R,\Hil)$ with respect to the $H^k$-norm given by $$\op{\xi}_k^2=\op{\partial_t^k\xi}^2+\op{\xi}^2,\,\,\|\cdot\|=\textrm{$L^2$-norm}$$ as in finite dimensions, we wish to conclude that the extended linear map $$i\partial_t:H^k_\phi(\R,\Hil)\to H^{k-1}_\phi(\R,\Hil)$$ defines an isomorphism of Hilbert spaces. Due to the small divisor problem in infinite dimensions, this is however no longer true. \par

Recalling that the complete basis $(z_n)_{n\in\Z^N}$ of $\Hil$ consists of eigenvectors of the time-$T$ map $\phi^A_T$ with corresponding eigenvalues $e^{i\epsilon_n T}$, observe that the (densly defined) linear operator $i\partial_t$ on $L^2_\phi(\R,\Hil)$ has a complete basis for eigenfunctions given by 
\[
e_{n,m}(t)=e^{-\epsilon_nit}e^{2\pi imt/T}z_n \,\,\text{for}\,\,m\in\Z, n\in\Z^N
\]
with  
\[
\op{e_{n,m}}_k^2=\left(\frac{2\pi m}{T}-\epsilon_n\right)^{2k}+1\,\,\textrm{and}\,\, \op{i\partial_t e_{n,m}}_{k-1}^2=\left(\frac{2\pi m}{T}-\epsilon_n\right)^{2k}+\left(\frac{2\pi m}{T}-\epsilon_n\right)^2;
\]
in particular, for $m=0$ we have that \[
\op{e_{n,0}}_k^2=\epsilon_n^{2k}+1\,\,\textrm{and}\,\, \op{i\partial_t e_{n,0}}_k^2=\epsilon_n^{2k}+\epsilon_n^2.
\]
As in our examples we need to assume that there is a small divisor problem, that is, there exists a subsequence $(\epsilon_n)$ with $\epsilon_n\to 0$ as $n\to\infty$. This in turn implies that \emph{there does not exist $c>0$ with} \[c\cdot\op{e_{n,0}}_k\leq \op{i\partial_t e_{n,0}}_{k-1}\,\,\textrm{for all}\,\,n\in\N.\]

This problem however can elegantly resolved by introducing the modified Hilbert space norm 
\[ \op{\xi}_{k,\sim}:= \op{i\partial_t\xi}_{k-1}=\op{\partial_t\xi}_{k-1},\,\,\textrm{in particular,}\,\,\op{e_{n,m}}_{k,\sim}^2=\left(\frac{2\pi m}{T}-\epsilon_n\right)^{2k}+\left(\frac{2\pi m}{T}-\epsilon_n\right)^2
\] on $C^{\infty}_\phi(\R,\Hil)$. Let $\wt{H}^k_\phi(\R,\Hil)$ denote the completion of  $C^{\infty}_\phi(\R,\Hil)$ with respect to the $\op{\cdot}_{k,\sim}$-norm. Note that for every $h'\in\R$ in an analogous manner one can define the modified Hilbert space $\wt{H}^k_\phi(\R,\Hil_{h'})$.
\begin{lemma}
The modified norm $\op{\cdot}_{k,\sim}$ defines a norm on $C^{\infty}_\phi(\R,\Hil_{h'})$ and $i\partial_t:\wt{H}^k_\phi(\R,\Hil_{h'})\to H^{k-1}_\phi(\R,\Hil_{h'})$ is a bounded linear isomorphism with bounded inverse.
\end{lemma}
\begin{proof} The fact that $\op{\cdot}_{k,\sim}$ is indeed a norm follows from the observation that $i\partial_t$ has no kernel on $C^{\infty}_\phi(\R,\Hil_{h'})$ which itself resulted from the fact that the linear symplectomorphism $\phi=\phi^A_T$ has no eigenvalue equal to one. The second statement follows from the enhanced statement that $i\partial_t$ even defines an isometry between $\wt{H}^k_\phi(\R,\Hil_{h'})$ and $H^{k-1}_\phi(\R,\Hil_{h'})$.
\end{proof}
In order to understand the relation between the different Hilbert spaces, we have the following lemma.
\begin{lemma}\label{contained-1}
For every $h'\in\R$ we have the inclusions  $H^k_\phi(\R,\Hil_{h'})\subset\wt{H}^k_\phi(\R,\Hil_{h'})$ and $\wt{H}^k_\phi(\R,\Hil_{h'})\subset H^k_\phi(\R,\Hil_{h'-h''})$ for every $h''>h_0$. In particular, $\wt{H}^k_\phi(\R,\Hil_{h'})\subset H^k_\phi(\R,\Hil)$ as long as $h'>h_0$.
\end{lemma}
\begin{proof} The first inclusion immediately follows from the fact that the map $i\partial_t:H^k_\phi(\R,\Hil_{h'})\to H^{k-1}_\phi(\R,\Hil_{h'})$ is bounded. For the second inclusion observe that the basis element $e_{n,m}$ has squared norm
\[|z_n|^2_{h'-h''}\left(\frac{2\pi m}{T}-\epsilon_n\right)^{2k}+|z_n|^2_{h'-h''}\,\,\textrm{in}\,\,H^k_\phi(\R,\Hil_{h'-h''})\]
and
\[|z_n|^2_{h'}\left(\frac{2\pi m}{T}-\epsilon_n\right)^{2k}+|z_n|^2_{h'}\left(\frac{2\pi m}{T}-\epsilon_n\right)^2\,\,\textrm{in}\,\,\wt{H}^k_\phi(\R,\Hil_{h'}).\] Since by \Cref{admissible-1} we know that $|\epsilon_n|>c\theta_n^{-h''}$, it follows that $|z_n|_{h'}|\epsilon_n|>c|z_n|_{h'-h''}$, which is enough to see that the second embedding is bounded. \end{proof}
Since we are interested in studying Floer curves with image in $M\times\Hil$, from now on we assume that $h'>h_0$.\par
For later use, we already mention the following useful result.
\begin{lemma}
\label{compactembedding1}
The space $H^k_\phi(\R,\Hil_{h})$ embeds compactly into $H^{k-1}_\phi(\R,\Hil_{h'})$ whenever $h>h'$.
\end{lemma}
We assume that, when $\Hil$ is replaced by $\R^{2n}$ and hence $\Hil_h=\Hil_{h'}=\R^{2n}$, this result is well-known to the reader. While $\Hil$, in contrast to $\R^{2n}$, no longer embeds compactly into itself, we instead use that $\Hil_h$ embeds compactly in $\Hil_{h'}$ for every $h>h'$. 
\begin{proof}
For the proof it suffices to observe that for the norms of each basis element $e_{n,m}$ in $H^k_\phi(\R,\Hil_{h})$ and $H^{k-1}_\phi(\R,\Hil_{h'})$ one has for all $n\in\Z^N,m\in\Z$ that
\[
|z_n|^2_{h'}\left(\frac{2\pi m}{T}-\epsilon_n\right)^{2k-2}+|z_n|^2_{h'}\leq \theta_n^{-2h+2h'}\cdot \left(|z_n|^2_h\left(\frac{2\pi m}{T}-\epsilon_n\right)^{2k}+|z_n|^2_h\right).
\]
\end{proof} 

For establishing the Fredholm property of the linearized Cauchy-Riemann operator $D_{\wt{u}}$ for every Floer curve $\wt{u}$ connecting $\phi^A_T$-periodic orbits $u^+$ and $u^-$, as in finite-dimensional Floer homology theory we first need to impose a suitable notion of nondegeneracy of the asymptotic orbits. Based on the introduction of the modified Hilbert space $\wt{H}^k_\phi(\R,\Hil)$ we give a modified definition of nondegeneracy. \par

\emph{From now on fix $h'\in\R$ with $(2d\leq) h_0< h'<h$.} Let $u=(u_M,u_{\Hil})\in H^k(S^1,M)\times\wt{H}^k_\phi(\R,\Hil_{h'})$ be a $\phi^A_T$-periodic solution to the Hamiltonian equation $\partial_tu=X^G_t(u)$, that is, of $J(u)\partial_tu=-\nabla G_t(u)$. Since $F_t$ is $h$-regularizing in the sense of \Cref{admissible-2}, we know by \Cref{gradient+Hessian} that the gradient $\nabla G(t,u)=\nabla G_t(u)$ of $G_t=F_t\circ\phi^A_{-t}$ defines a $\floor{h/d}$-times continuously differentiable map $\nabla G:\R\times\Hil\to\Hil_{+h}$. In order to ensure that the gradient $u\mapsto\nabla G_t(u)$ defines a section in the Hilbert space bundle over $H^k(S^1,M)\times H^k_\phi(\R,\Hil)$ with fibre $H^k(u_M^*TM)\times H^k_\phi(\R,\Hil_{+h})$ using \cite[Prop. B.1.19]{MDSa}, \emph{from now on we will assume that $1<k\leq\floor{h/d}$,} where we use that $\floor{h/d}\geq 2$ as $h>h_0\geq 2d$ by definition. Since $\wt{H}^k_\phi(\R,\Hil_{h'})\subset H^k_\phi(\R,\Hil)$ and $H^k_\phi(\R,\Hil_{+h})\subset H^{k-1}_\phi(\R,\Hil_{h'})$ as long as $h_0 < h'< h$, we find that the Hamilton operator $u\mapsto J(u)\partial_t u+\nabla G_t(u)$ defines a section in the Hilbert space bundle over $H^k(S^1,M)\times \wt{H}^k_\phi(\R,\Hil_{h'})$ with fibre $H^{k-1}(u_M^*TM)\times H^{k-1}_\phi(\R,\Hil_{h'})$. Linearizing the Hamilton operator at $u\in H^k(S^1,M)\times \wt{H}^k_\phi(\R,\Hil_{h'})$, we obtain a linear operator from $H^k(u_M^*TM)\oplus \widetilde{H}^k_{\phi}(\R,\Hil_{h'})$ to $H^{k-1}(u_M^*TM)\oplus H^{k-1}_{\phi}(\R,\Hil_{h'})$ which, after choosing a unitary trivialization of $u_M^*TM\to S^1$, is given by a linear operator \[i\partial_t+S_u: H^k(S^1,\R^{\dim M})\oplus\widetilde{H}^k_{\phi}(\R,\Hil_{h'})\to H^{k-1}(S^1,\R^{\dim M})\oplus H^{k-1}_{\phi}(\R,\Hil_{h'}).\] Here $S_u$ is a multiplication operator defined pointwise by $S_u(t)\in (\R^{\dim M}\oplus\Hil_{-h})^*\otimes (\R^{\dim M}\oplus\Hil_{+h})$, given by the Hessian $\nabla^2 G_t(u(t))\in (T_{u_M(t)}M\oplus\Hil_{-h})^*\otimes (T_{u_M(t)}M\oplus\Hil_{+h})$ plus extra terms depending on the almost complex structure $J_M$ on $M$ and the chosen unitary trivialization of $u_M^*TM\to S^1$.\par
In order to prove that $i\partial_t+S_u$ indeed defines a bounded linear operator from $H^k(S^1,\R^{\dim M})\oplus\widetilde{H}^k_{\phi}(\R,\Hil_{h'})$ to $H^{k-1}(S^1,\R^{\dim M})\oplus H^{k-1}_{\phi}(\R,\Hil_{h'})$ and to show the independence on the choice of $0\leq h'<h$ and $1<k\leq\floor{h/d}$, we start with the following regularity result. 
\begin{lemma}
\label{Hklemma1}
Every $\phi^A_T$-periodic orbit $u$ of $G_t$ is an element of $H^{\floor{h/d}+1}(S^1,M)\times\wt{H}^{\floor{h/d}+1}_\phi(\R,\Hil_{h})$ with $\widetilde{H}^{\floor{h/d}+1}_{\phi}(\R,\Hil_{h})\subset H^{\floor{h/d}+1}_{\phi}(\R,\Hil_{h-h''})$ for every $h''>h_0$, in particular, $S_u$ defines a bounded linear operator from $H^k(S^1,\R^{\dim M})\oplus H^k_\phi(\R,\Hil_{-h})$ to $H^k(S^1,\R^{\dim M})\oplus H^k_\phi(\R,\Hil_{+h})$ when $1<k\leq\floor{h/d}$.
\end{lemma}
\begin{proof}
From $u\in H^k(S^1,M)\times\wt{H}^k_\phi(\R,\Hil_{h'})\subset H^k(S^1,M)\times H^k_\phi(\R,\Hil)$ it follows that $\nabla G_t(u)\in H^k(u_M^*TM)\oplus H^k_\phi(\R,\Hil_{+h})$ which in turn gives $J(u)\partial_t u\in H^k(u_M^*TM)\oplus H^k_\phi(\R,\Hil_{+h})$. Using the fact that $i\partial_t$ defines a bijection between $\wt{H}^{k+1}_\phi(\R,\Hil_{+h})$ and $H^k_\phi(\R,\Hil_{+h})$, we find that $u\in H^{k+1}(S^1,M)\times\wt{H}^{k+1}_\phi(\R,\Hil_{+h})$ with $\wt{H}^{k+1}_\phi(\R,\Hil_{+h})\subset H^{k+1}_\phi(\R,\Hil_{+h-h''})$ as long as $h''>h_0$ by \Cref{contained-1}. This reasoning can be iterated as long as $k\leq\floor{h/d}$. Now we can use \cite[Prop. B.1.19]{MDSa} on composition to conclude that the $H^k_\phi(\R,(\R^{\dim M}\oplus\Hil_{-h})^*\otimes(\R^{\dim M}\oplus\Hil_{+h}))$-norm of $S_u$ is finite as long as $k\leq\floor{h/d}$. Finally employing \cite[Prop. B.1.19]{MDSa} on products, we find that $S_u$ defines a bounded linear operator from $H^k_\phi(\R,\R^{\dim M}\oplus\Hil_{-h})$ to $H^k_\phi(\R,\R^{\dim M}\oplus\Hil_{+h})$ as claimed.
\end{proof}
Note that, after setting $h''=\frac{1}{2}(h+h_0)$, we find that $u\in C^{\floor{h/d}}_\phi(\R,M\times\Hil_{\frac{1}{2}(h-h_0)})$; in particular, $u\in C^{\infty}_\phi(\R,M\times\Hil_\infty)$ in the case when $G$ is $h$-regularizing for \emph{all} $h>h_0$.\par 
Since the embedding $H^k_\phi(\R,\Hil_h)\subset H^{k-1}_\phi(\R,\Hil_{h'})$ is compact as long as $h'<h$, we can even deduce

\begin{lemma} 
$S_u:H^k(S^1,\R^{\dim M})\oplus\wt{H}^k_\phi(\R,\Hil_{h'})\to H^{k-1}(S^1,\R^{\dim M})\oplus H^{k-1}_\phi(\R,\Hil_{h'})$ is a compact operator.
\end{lemma}

There clearly also exists a corresponding statement for the linearization. 
\begin{proposition}\label{linear-regularity-1}
The kernel of the linearized operator $i\partial_t+S_u:\wt{H}^k_\phi(\R,\R^{\dim M}\oplus\Hil_{h'})=H^k(S^1,\R^{\dim M})\oplus\wt{H}^k_\phi(\R,\Hil_{h'})\to H^{k-1}_\phi(\R,\R^{\dim M}\oplus\Hil_{h'})$ is contained in 
$H^{\floor{h/d}+1}(S^1,M)\oplus\wt{H}^{\floor{h/d}+1}_\phi(\R,\Hil_h)$. \end{proposition}
\begin{proof} From $\xi\in\wt{H}^k_\phi(\R,\R^{\dim M}\oplus\Hil_{h'})\subset H^k_\phi(\R,\R^{\dim M}\oplus\Hil_{-h})$ it follows that $S_u\cdot\xi\in H^k_\phi(\R,\R^{\dim M}\oplus\Hil_{+h})$ which in turn gives $i\partial_t\xi\in H^k_\phi(\R,\R^{\dim M}\oplus\Hil_{+h})$. Using the fact that $i\partial_t$ defines a bijection between $\wt{H}^{k+1}_\phi(\R,\Hil_{+h})$ and $H^k_\phi(\R,\Hil_{+h})$, we find that $\xi\in \wt{H}^{k+1}_\phi(\R,\R^{\dim M}\oplus\Hil_{+h})\subset H^{k+1}_\phi(\R,\R^{\dim M}\oplus\Hil_{+h-h''})$ for every $h''>h_0$ by \Cref{contained-1}. This reasoning can be iterated as long as $k\leq\floor{h/d}$. \end{proof}

The following definition is the natural adaption of the concept of nondegeneracy from finite-dimensional Floer theory.
\begin{definition}\label{nondegenerate}
A $\phi^A_T$-periodic orbit $u\in\P(\phi,G)$ is called \emph{nondegenerate up to small divisors} if the linear map
\[
i\partial_t+S_u: H^k(S^1,\R^{\dim M})\oplus\wt{H}^k_\phi(\R,\Hil_{h'})\to H^{k-1}(S^1,\R^{\dim M})\oplus H^{k-1}_\phi(\R,\Hil_{h'})\]
is bounded with bounded inverse.
\end{definition}

In analogy with the finite-dimensional case, it is immediate to check that nondegeneracy of the $\phi^A_T$-periodic orbit of $G_t$ can be reformulated as a statement about eigenvalues of the linearization $T_{u(0)}\phi^H_T:\Hil\to\Hil$ of the time-$T$ flow map $\phi^H_T=\phi^G_T\circ\phi^A_T$ at $u(0)$.
\begin{proposition} The property in \Cref{nondegenerate} is independent of the choice for $1<k\leq\floor{h/d}$ and $h_0 < h' < h$. Furthermore it is equivalent to the fact that the linearized return map $T_{u(0)}\phi^H_T:\Hil\to\Hil$ has no eigenvalue equal to one.
\end{proposition}

\begin{proof}
Since $i\partial_t+S_u$ is a Fredholm operator of index zero, it follows that $i\partial_t+S_u$ has a bounded inverse if and only if it is injective. Since the kernel of $i\partial_t+S_u$ is independent of the choice of $h_0 < h'<h$ and $1<k\leq\floor{h/d}$ by \Cref{linear-regularity-1}, the independence on the choice for $k$ and for $h'$ follows. Furthermore it follows as in finite-dimensional Floer theory that the condition about the Hessian of the symplectic action functional from \Cref{nondegenerate} is equivalent to the statement about eigenvalues of the linearized return map.
\end{proof}
Note that, in contrast to the case of finite-dimensional Floer theory, here this condition also encompasses that the eigenvalues of $T_{u(0)}\phi^H_T$ do not converge to one faster than the eigenvalues of $\phi^A_T$. As in the finite-dimensional case we prove below that for a generic choice of admissible $F_t$ (and hence of $h$-regularizing $G_t$) every $\phi^A_T$-periodic orbit of $G_t$ is nondegenerate up to small divisors. \\\par

Let $u^\pm=(u^\pm_M,u^\pm_{\Hil}):\R\to \wt{M}_{h'}=M\times\Hil_{h'}$ be  $\phi^A_T$-periodic orbits for $G_t$ which we are assumed to be nondegenerate up to small divisors in the sense of \Cref{nondegenerate}. Consider the space $C^{\infty}_\phi(u^+,u^-)$ of all smooth maps $\wt{u}=(\wt{u}_M,\wt{u}_{\Hil}):\R^2\to M\times\Hil_{h'}$ satisfying the periodicity condition $\wt{u}(s,t+T)=\phi^{-1}(\wt{u}(s,t))$ and the asymptotic conditions $\wt{u}(s,\cdot)\to u^\pm$ as $s\to \pm\infty$. Note that it can be written as a product,\[C^{\infty}_\phi(u^+,u^-)=C^{\infty}(u^+_M,u^-_M)\times C^{\infty}_\phi(u^+_{\Hil},u^-_{\Hil}),\] with the affine linear space $C^{\infty}_{\phi}(u^+_{\Hil},u^-_{\Hil})=\{\wt{u}_{\Hil}^*\}+C^{\infty}_{\phi,0}(\R^2,\Hil_{h'})$, where $\wt{u}_{\Hil}^*\in C^{\infty}_{\phi}(u^+_{\Hil},u^-_{\Hil})$ is a reference map and the maps in $C^{\infty}_{\phi,0}(\R^2,\Hil_{h'})$ converge to zero as $s\to\pm\infty$.\par
Recall that, following the case of finite-dimensional Floer theory, at the first hand we would like to prove that the linearized Cauchy-Riemann operator $D_{\wt{u}}$ is a Fredholm operator when viewed as a linear map from $H^k(\wt{u}_M^*TM)\oplus H^k_\phi(\R^2,\Hil_{h'})$ to $H^{k-1}(\wt{u}_M^*TM)\oplus H^{k-1}_\phi(\R^2,\Hil_{h'})$. However, due to the small divisor problem, it is again clear that the Hilbert space $H^k_\phi(\R^2,\Hil_{h'})$ needs to be replaced by a modified Hilbert space completion $\wt{H}^k_\phi(\R^2,\Hil_{h'})$ of $C^{\infty}_{0,\phi}(\R^2,\Hil_{h'})$. In analogy with the norm $\|\cdot\|_{k,\sim}$ used to define the Hilbert space completion $\wt{H}^k_\phi(\R,\Hil_{h'})$ of $C^\infty_\phi(\R,\Hil_{h'})$, we now define the modified Hilbert space norm 
\[\op{\wt{\xi}}_{k,\sim}:=\op{\CR\wt{\xi}}_{k-1}\,\,\textrm{with the $H^{k-1}$-norm for maps $\wt{\xi}:\R^2\to\Hil_{h'}$, $\wt{\xi}(s,t+T)=\phi^{-1}\cdot\wt{\xi}(s,t)$}\] and define $\wt{\xi}\in\wt{H}^k_\phi(\R^2,\Hil_{h'})$ whenever $\|\wt{\xi}\|_{k,\sim}$ is finite. The fact that $\op{\cdot}_{k,\sim}$ is indeed a norm again follows from the observation that $\CR=\partial_s+i\partial_t$ has no kernel on $C^{\infty}_{0,\phi}(\R^2,\Hil_{h'})$ which itself results from the fact that $i\partial_t$ has no kernel on $C^{\infty}_{0,\phi}(\R,\Hil_{h'})$. Note that we have  \[\widetilde{\xi}\in\widetilde{H}^k_{\phi}(\R^2,\Hil_{h'})\,\,\text{if and only if}\,\, \CR\widetilde{\xi}\in H^{k-1}_{\phi}(\R^2,\Hil_{h'}).\] 
First we prove the following analogue of \Cref{contained-1}. 
\begin{lemma}\label{contained-2}
We have $\wt{H}^k_\phi(\R^2,\Hil_{h'})\subset H^k_\phi(\R^2,\Hil_{h'-h''})$ for every $h''>h_0$. In particular, $\wt{H}^k_\phi(\R^2,\Hil_{h'})\subset H^k_\phi(\R^2,\Hil)$ as long as $h'>h_0$.
\end{lemma}
\begin{proof}
Consider $\wt{\xi}:\R^2\to\Hil$ satisfying the periodicity condition $\wt{\xi}(s,t+T)=\phi^A_{-T}\cdot\wt{\xi}(s,t)$ and the asymptotic condition $\wt{\xi}(s,t)\to 0$ as $s\to\pm\infty$. Expand $\wt{\xi}$ and $\wt{\eta}=\CR\wt{\xi}=\partial_s\wt{\xi}+i\partial_t\wt{\xi}$ as Fourier series $$\wt{\xi}(s)=\sum_{n,m}\widehat{\wt{\xi}(s)}(n,m)\cdot e_{n,m},\,\wt{\eta}(s)=\sum_{n,m}\widehat{\wt{\eta}(s)}(n,m)\cdot e_{n,m}$$ with respect to the complete basis $(e_{n,m})_{n,m}$ of eigenfunctions for the operator $i\partial_t$ on $L^2_\phi(\R,\Hil)$ given by $e_{n,m}(t)=e^{-\epsilon_nit}e^{2\pi imt/T}z_n$ for $n\in\Z^N$, $m\in\Z$. Then it follows that the $s$-dependent Fourier coefficients $f_{n,m}(s)=\widehat{\wt{\xi}(s)}(n,m)$, $g_{n,m}(s)=\widehat{\wt{\eta}(s)}(n,m)$ satisfy the ordinary differential equation $$(\partial_s+\lambda_{n,m})\cdot f_{n,m}(s)=g_{n,m}(s)\,\,\textrm{with the eigenvalues}\,\,\lambda_{n,m}=\frac{2\pi m}{T}-\epsilon_n\neq 0\,\,\text{of}\,\,i\partial_t.$$ Taking also the asymptotic conditions $f_{n,m}(s),g_{n,m}(s)\to 0$ as $s\to\pm\infty$ into account, it follows that $g_{n,m}$ can be explicitly computed from $f_{n,m}$ by taking convolution with the kernel function $K_{n,m}$ given by $K_{n,m}(s)=\exp(\lambda_{n,m}\cdot s)$ for $s\leq 0$, $K_{n,m}(s)=0$ for $s>0$ if $\lambda_{n,m}>0$ and $K_{n,m}(s)=\exp(\lambda_{n,m}\cdot s)$ for $s\geq 0$, $K_{n,m}(s)=0$ for $s<0$ if $\lambda_{n,m}<0$. Using Young's inequality for convolutions it follows for the $L^2$-norms of $f_{n,m}$ and $g_{n,m}$ that $|\lambda_{n,m}|\cdot \|f_{n,m}\|_2\leq \|g_{n,m}\|_2$ with $|\lambda_{n,m}||z_n|_{h'}\geq|\epsilon_n||z_n|_{h'}\geq c''\cdot|z_n|_{h'-h''}$ for all $n\in\Z^N$, $m\in\Z$.
Using in addition that $\|\partial_s f_{n,m}\|_2\leq |\lambda_{n,m}|\cdot\|f_{n,m}\|_2+\|g_{n,m}\|_2\leq 2\|g_{n,m}\|_2$, we find that the $H^k_{\phi}(\R^2,\Hil_{h'-h''})$-norm of $\wt{\xi}$ is finite when the $H^{k-1}_{\phi}(\R^2,\Hil_{h'})$-norm of $\wt{\eta}=\CR\wt{\xi}$ is finite. 
\end{proof}

For each $u^+,u^-\in\P(\phi,G)$ we define the moduli space $\M(u^+,u^-;G)$ as the zero set of the nonlinear Cauchy-Riemann operator \[\CR_G\wt{u}=\partial_s\wt{u}+J(\wt{u})\partial_t\wt{u}+\nabla G_t(\wt{u})\,\,\text{in}\,\,\B=\B(u^+,u^-)=H^k(u_M^+,u_M^-)\times(\{\wt{u}_{\Hil}^*\}+\wt{H}^k_\phi(\R^2,\Hil_{h'})).\] Let $\wt{u}=(\wt{u}_M,\wt{u}_{\Hil})\in\M(u^+,u^-;G)$. Recall that, since $F_t$ is $h$-regularizing in the sense of \Cref{admissible-2}, we know by \Cref{gradient+Hessian} that the gradient $\nabla G(t,u)=\nabla G_t(u)$ of $G_t=F_t\circ\phi^A_{-t}$ defines a $\floor{h/d}$-times continuously differentiable map $\nabla G:\R\times\Hil\to\Hil_{+h}$. Using \cite[Prop. B.1.19]{MDSa} it again follows that the gradient $\wt{u}\mapsto\nabla G_t(\wt{u})$ defines a section in the Hilbert space bundle over $H^k(u_M^+,u_M^-)\times(\{\wt{u}_{\Hil}^*\}+H^k_\phi(\R^2,\Hil))$ with fibre $H^k(\wt{u}_M^*TM)\times H^k_\phi(\R^2,\Hil_{+h})$ using we assume $1<k\leq\floor{h/d}$. Since $\wt{H}^k_\phi(\R^2,\Hil_{h'})\subset H^k_\phi(\R^2,\Hil)$ by \Cref{contained-2} and $H^k_\phi(\R^2,\Hil_{+h})\subset H^{k-1}_\phi(\R^2,\Hil_{h'})$ as long as $h_0 < h'< h$, we find that the perturbed Cauchy-Riemann operator $\wt{u}\mapsto \partial_s\wt{u}+J(\wt{u})\partial_t \wt{u}+\nabla G_t(\wt{u})$ defines a section in the Hilbert space bundle \[\E=\E(u^+,u^-)\to\B(u^+,u^-)\,\,\text{with fibre}\,\,\E_{\wt{u}}=H^{k-1}(\wt{u}_M^*TM)\times H^{k-1}_\phi(\R^2,\Hil_{h'})\,\,\text{over}\,\,\wt{u}\in\B.\] Linearizing the perturbed Cauchy-Riemann operator at $\wt{u}=(\wt{u}_M,\wt{u}_{\Hil})$, we obtain a linear operator from $T_{\wt{u}}\B=H^k(\wt{u}_M^*TM)\oplus \widetilde{H}^k_{\phi}(\R^2,\Hil_{h'})$ to $\E_{\wt{u}}=H^{k-1}(\wt{u}_M^*TM)\oplus H^{k-1}_{\phi}(\R^2,\Hil_{h'})$ which, after choosing a unitary trivialization of $\wt{u}_M^*TM\to \R\times S^1$, is given by a linear operator \[D_{\wt{u}}=\CR+S_{\wt{u}}: H^k(\R\times S^1,\R^{\dim M})\oplus\widetilde{H}^k_{\phi}(\R^2,\Hil_{h'})\to H^{k-1}(\R\times S^1,\R^{\dim M})\oplus H^{k-1}_{\phi}(\R^2,\Hil_{h'}).\] Here $S_{\wt{u}}$ is a multiplication operator defined pointwise by $S_{\wt{u}}(s,t)\in (\R^{\dim M}\oplus\Hil_{-h})^*\otimes (\R^{\dim M}\oplus\Hil_{+h})$, given by the Hessian $\nabla^2 G_t(\wt{u}(s,t))\in (T_{\wt{u}_M(s,t)}M\oplus\Hil_{-h})^*\otimes (T_{\wt{u}_M(s,t)}M\oplus\Hil_{+h})$ plus extra terms depending on the almost complex structure $J_M$ on $M$ and the chosen unitary trivialization of $\wt{u}_M^*TM\to\R\times S^1$. Again, since $G_t$ is only $\floor{h/d}$-times continuously differentiable with respect to $t$, we again need to assume that $1<k\leq\floor{h/d}$.\par
In order to prove that $D_{\wt{u}}=\CR+S_{\wt{u}}$ indeed defines a bounded linear operator from $H^k(\R\times S^1,\R^{\dim M})\oplus\widetilde{H}^k_{\phi}(\R^2,\Hil_{h'})$ to $H^{k-1}(\R\times S^1,\R^{\dim M})\oplus H^{k-1}_{\phi}(\R^2,\Hil_{h'})$ and to show the independence on the choice of $h_0 < h'<h$ and $1<k\leq\floor{h/d}$, we again need a regularity result.

\begin{lemma}\label{nonlinear-regularity}
\sloppy The moduli space $\M(u^+,u^-;G)=\CR_G^{-1}(0)$ is contained in $H^{\floor{h/d}+1}(u_M^+,u_M^-)\times (\{\wt{u}_{\Hil}^*\}+\wt{H}^{\floor{h/d}+1}_\phi(\R^2,\Hil_h))$ with $\wt{H}^{\floor{h/d}+1}_\phi(\R^2,\Hil_h)\subset H^{\floor{h/d}+1}_\phi(\R^2,\Hil_{h-h''})$ for every $h''>h_0$, in particular, $S_{\wt{u}}$ defines a bounded linear operator from $H^k(\R\times S^1,\R^{\dim M})\oplus H^k_\phi(\R^2,\Hil_{-h})$ to $H^k(\R\times S^1,\R^{\dim M})\oplus H^k_\phi(\R^2,\Hil_{+h})$ when $1<k\leq\floor{h/d}$.
\end{lemma}
\begin{proof}
From $\wt{u}\in H^k(u_M^+,u_M^-)\times (\{u_{\Hil}^*\}+ \wt{H}^k_\phi(\R^2,\Hil_{h'}))$ with $\wt{H}^k_\phi(\R^2,\Hil_{h'})\subset H^k_\phi(\R^2,\Hil_{-h})$ it follows that $\nabla G_t(\wt{u})\in H^k(\wt{u}_M^*TM)\oplus H^k_\phi(\R^2,\Hil_{+h})$ which in turn gives $\partial_s\wt{u}+J(\wt{u})\partial_t\wt{u}\in H^k(\wt{u}_M^*TM)\oplus H^k_\phi(\R^2,\Hil_{+h})$. Elliptic regularity in finite dimensions combined with the fact that $\CR=\partial_s+i\partial_t$ defines a bijection between $\wt{H}^{k+1}_\phi(\R^2,\Hil_{+h})$ and $H^k_\phi(\R^2,\Hil_{+h})$, we find that $\wt{u}\in H^{k+1}(u_M^+,u_M^-)\times (\{u_{\Hil}^*\}+\wt{H}^{k+1}_\phi(\R^2,\Hil_{+h}))$ with $\wt{H}^{k+1}_\phi(\R^2,\Hil_{+h})\subset H^{k+1}_\phi(\R^2,\Hil_{+h-h''})$. This reasoning can be iterated as long as $k\leq\floor{h/d}$. Using again \cite[Prop. B.1.19]{MDSa}, we can conclude not only that the $H^k_\phi(\R^2,(\R^{\dim M}\oplus\Hil_{-h})^*\otimes(\R^{\dim M}\oplus\Hil_{+h})$-norm of $S_{\wt{u}}$ is finite but also that $S_{\wt{u}}$ defines a bounded linear operator from $H^k(\R\times S^1,\R^{\dim M})\oplus H^k_\phi(\R^2,\Hil_{-h})$ to $H^k(\R\times S^1,\R^{\dim M})\oplus H^k_\phi(\R^2,\Hil_h)$.
\end{proof}
It follows that $S_{\wt{u}}$ defines a bounded linear operator from $H^k(\R\times S^1,\R^{\dim M})\oplus \wt{H}^k_\phi(\R^2,\Hil_{-h})$ to $H^{k-1}(\R\times S^1,\R^{\dim M})\oplus H^{k-1}_\phi(\R^2,\Hil_{+h})$. Note again that, after setting $h''=\frac{1}{2}(h+h_0)$, we find that $\wt{u}\in C^{\floor{h/d}-1}_\phi(\R^2,M\times\Hil_{\frac{1}{2}(h-h_0)})$; in particular, $\wt{u}\in C^{\infty}_\phi(\R^2,M\times\Hil_\infty)$ in the case when $G$ is $h$-regularizing for \emph{all} $h>h_0$. 
There clearly also exists a corresponding statement for the linearization $D_{\wt{u}}$ of $\CR_G$. 
\begin{proposition}\label{linear-regularity}
The kernel of the linearized operator $D_{\wt{u}}:\wt{H}^k_\phi(\R^2,\R^{\dim M}\oplus\Hil_{h'})=H^k(\R\times S^1,\R^{\dim M})\oplus\wt{H}^k_\phi(\R^2,\Hil_{h'})\to H^{k-1}_\phi(\R^2,\R^{\dim M}\oplus\Hil_{h'})$ is contained in 
$H^{\floor{h/d}+1}(\R\times S^1,\R^{\dim M})\oplus\wt{H}^{\floor{h/d}+1}_\phi(\R^2,\Hil_h)$. \end{proposition}
\begin{proof} From $\wt{\xi}\in\wt{H}^k_\phi(\R^2,\R^{\dim M}\oplus\Hil_{h'})\subset H^k_\phi(\R^2,\R^{\dim M}\oplus\Hil_{-h})$ it follows that $S_{\wt{u}}\cdot\wt{\xi}\in H^k_\phi(\R^2,\R^{\dim M}\oplus\Hil_{+h})$ which in turn gives $\CR\wt{\xi}\in H^k_\phi(\R^2,\R^{\dim M}\oplus\Hil_{+h})$. Using linear ellipticity and the fact that $\CR$ defines a bijection between $\wt{H}^{k+1}_\phi(\R^2,\Hil_{+h})$ and $H^k_\phi(\R^2,\Hil_{+h})$, we find that $\wt{\xi}\in \wt{H}^{k+1}_\phi(\R^2,\R^{\dim M}\oplus\Hil_{+h})\subset H^{k+1}_\phi(\R^2,\R^{\dim M}\oplus\Hil_{+h-h''})$ for every $h''>h_0$ by \Cref{contained-2}. This reasoning can be iterated as long as $k\leq\floor{h/d}$. \end{proof}

We now state the main result of this paper. 
\begin{theorem}\label{Fredholm}
Fix $1< k\leq\floor{h/d}$ and $h_0 < h'<h$ and assume that $u^+$ and $u^-$ are two $\phi$-periodic orbits of $G_t$ which are nondegenerate up to small divisors in the sense of \Cref{nondegenerate}. Then for every Floer curve $\wt{u}\in \M(u^+,u^-;G)$ the linearized Cauchy-Riemann operator \[D_{\wt{u}}:H^k(\R\times S^1,\R^{\dim M})\oplus\wt{H}^k_\phi(\R^2,\Hil_{h'})\to H^{k-1}(\R\times S^1,\R^{\dim M})\oplus H^{k-1}_\phi(\R^2,\Hil_{h'})\] is a linear Fredholm operator. In other words, the nonlinear Cauchy-Riemann operator $\CR_G\wt{u}=\partial_s\wt{u}+J(\wt{u})\partial_t\wt{u}+\nabla G_t(\wt{u})$, viewed as section in the Hilbert space bundle $\E=\E(u^+,u^-)\to\B(u^+,u^-)$ with fibre $\E_{\wt{u}}=H^{k-1}(\wt{u}_M^*TM)\oplus H^{k-1}_{\phi}(\R^2,\Hil_{h'})$ over $\widetilde{u}\in\B=\B(u^+,u^-)=\widetilde{H}^k_{\phi}(u^+,u^-)=H^k(u_M^+,u_M^-)\times(\{\wt{u}_{\Hil}^*\}+\wt{H}^k_\phi(\R^2,\Hil_{h'}))$, is nonlinear Fredholm.
\end{theorem}

Assuming that $u^+$ and $u^-$ meet the requirements of \Cref{nondegenerate}, as in the case of finite-dimensional symplectic manifolds we will prove \Cref{Fredholm} by establishing the semi-Fredholm property for $D_{\wt{u}}$ and for its adjoint operator $D_{\wt{u}}^*$. While one might expect the proof to be fully analogous to the finite-dimensional proof, it however turns out that special attention needs to be paid to the infinite-dimensional tail. Indeed we need to make use of the following approximation result.
\begin{lemma}
\label{findimest}
\label{fdlemma}
Let $u$, $u^\pm$ be $\phi^A_T$-periodic orbits which is non-degenerate up to small divisors and $\wt{u}$ a Floer curve in $\M(u^+,u^-)$. Let $S_{{u}}^\ell$, $S_{\wt{u}}^\ell$ be the composition of $S_{{u}}$, $S_{\wt{u}}$ with the projection  
onto the finite-dimensional Euclidean subspace $\Hil^{\ell}\subset\Hil_{h'}$ with canonical basis $|z_n|_{h'}^{-1}z_n$, $|n|\leq\ell$. Then for every $1< k\leq\floor{h/d}$ and $h_0 < h'<h$ we have
\begin{enumerate}
    \item $\op{S_{{u}}-S_{{u}}^\ell}\ell^{2h-1}\to0$ as $\ell\to\infty$, where the norm is the operator norm for operators on $H^k_\phi(\R^n,\R^{\dim M}\oplus\Hil_{h'})$ for $n=1,2$. Furthermore, $S_u^\ell\to S_u$ in the operator norm for operators $\wt{H}^k_\phi(\R^n,\R^{\dim M}\oplus\Hil_{h'})\to H^{k-1}_\phi(\R^n,\R^{\dim M}\oplus\Hil_{h'})$ for $n=1,2$.
    \item $\op{S_{\wt{u}}-S_{\wt{u}}^\ell}\ell^{2h-1}\to0$ as $\ell\to\infty$, where the norm is the operator norm for operators on $H^k_\phi(\R^2,\R^{\dim M}\oplus\Hil_{h'})$. Furthermore, $S_{\wt{u}}^\ell\to S_{\wt{u}}$ in the operator norm for operators $\wt{H}^k_\phi(\R^2,\R^{\dim M}\oplus\Hil_{h'})\to H^{k-1}_\phi(\R^2,\R^{\dim M}\oplus\Hil_{h'})$.
\end{enumerate}
\end{lemma}
Note that here we use the abbreviations $\wt{H}^k_\phi(\R,\R^{\dim M}\oplus\Hil_{h'})=H^k(S^1,\R^{\dim M})\oplus\wt{H}^k_\phi(\R,\Hil_{h'})$ and $\wt{H}^k_\phi(\R^2,\R^{\dim M}\oplus\Hil_{h'})=H^k(\R\times S^1,\R^{\dim M})\oplus\wt{H}^k_\phi(\R^2,\Hil_{h'})$.
\begin{proof}
Recalling from the proof of \Cref{Hklemma1} that the linear operator $S_u$ is given by multiplication with $S_u\in H^k_\phi(\R,(\R^{\dim M}\oplus\Hil_{-h})^*\otimes(\R^{\dim M}\oplus\Hil_{+h}))$, we observe that
\[
\op{S_u-S_u^\ell}\to0\quad\text{as}\quad\ell\to\infty\quad\text{in}\quad H^k_\phi(\R,(\R^{\dim M}\oplus\Hil_{-h})^*\otimes(\R^{\dim M}\oplus\Hil_{+h}))
\]
which in turn immediately implies that
\[
\op{S_u-S_u^\ell}\ell^{2h-1}\to0\quad\text{as}\quad\ell\to\infty\quad\text{in}\quad H^k_\phi(\R,(\R^{\dim M}\oplus\Hil_{h'})^*\otimes(\R^{\dim M}\oplus\Hil_{h'})),
\]
where we use that $\displaystyle|z_n^*|_{-h}\cdot |z_n|_{+h}=\theta_n^{2h}\cdot |z_n^*|_{h'}\cdot |z_n|_{h'}$ with $\theta_n^{2h}=O(|n|^{2h})$. For the corresponding statements in the operator norm it suffices to apply the multiplication inequality of \cite[Prop. B.1.19]{MDSa}. Using that $\wt{H}^k_\phi(\R^n,\Hil_{h'})$ embeds in $H^k_\phi(\R^n,\Hil_{-h})$ and $H^k_\phi(\R^n,\Hil_{+h})\subset H^{k-1}_\phi(\R^n,\Hil_{h'})$, the statement in the operator norm for operators $\wt{H}^k_\phi(\R^n,\R^{\dim M}\oplus\Hil_{h'})\to H^{k-1}_\phi(\R^n,\R^{\dim M}\oplus\Hil_{h'})$ follows as well for $n=1,2$. 
For the second statement, recalling from the proof of \Cref{nonlinear-regularity} that the linear operator $S_{\wt{u}}$ is given by multiplication with $S_{\wt{u}}\in H^k_\phi(\R^2,(\R^{\dim M}\oplus\Hil_{-h})^*\otimes(\R^{\dim M}\oplus\Hil_{+h}))$, we observe that
\[
\op{S_{\wt{u}}-S_{\wt{u}}^\ell}\to0\quad\text{as}\quad\ell\to\infty\quad\text{in}\quad H^k_\phi(\R^2,(\R^{\dim M}\oplus\Hil_{-h})^*\otimes(\R^{\dim M}\oplus\Hil_{+h}))
\]
which again immediately implies that
\[
\op{S_{\wt{u}}-S_{\wt{u}}^\ell}\ell^{2h-1}\to0\quad\text{as}\quad\ell\to\infty\quad\text{in}\quad H^k_\phi(\R^2,(\R^{\dim M}\oplus\Hil_{h'})^*\otimes(\R^{\dim M}\oplus\Hil_{h'})).
\]
While for the corresponding statements in the operator norm it again suffices to apply the multiplication inequality of \cite[Prop. B.1.19]{MDSa}, the statement about the operator norm for operators $\wt{H}^k_\phi(\R^2,\R^{\dim M}\oplus\Hil_{h'})\to H^{k-1}_\phi(\R^2,\R^{\dim M}\oplus\Hil_{h'})$ follows again using that $\wt{H}^k_\phi(\R^2,\Hil_{h'})\subset H^k_\phi(\R^2,\Hil_{-h})$ and $H^k_\phi(\R^2,\Hil_{h})\subset H^{k-1}_\phi(\R^2,\Hil_{h'})$. \end{proof}

We first need the following injectivity result.
\begin{lemma}
\label{injectivityestimate}
Assume that the $\phi^A_T$-periodic orbit of $G_t$ is nondegenerate up to small divisors in the sense of \Cref{nondegenerate}. Then for every $1< k\leq\floor{h/d}$ and $h_0 < h'<h$ the linear operator 
\[
\overline{\partial}+S_u:H^k(\R\times S^1,\R^{\dim M})\oplus\wt{H}^k_\phi(\R^2,\Hil_{h'})\to H^{k-1}(\R\times S^1,\R^{\dim M})\oplus H^{k-1}_\phi(\R^2,\Hil_{h'})
\]
is bounded with bounded inverse.
\end{lemma}
\begin{proof} 
Assume that the operator $i\partial_t+S_u:H^k(\R\times S^1,\R^{\dim M})\oplus\wt{H}^k_\phi(\R,\Hil_{h'})\to H^{k-1}(\R\times S^1,\R^{\dim M})\oplus H^{k-1}_\phi(\R,\Hil_{h'})$ is bounded with bounded inverse. Write $\Hil_{h'}=\Hil^\ell\oplus\Hil^\ell_\perp$ where $\Hil^\ell=\Span_{|n|\leq\ell}\{|z_n|_{h'}^{-1}z_n\}\subset\Hil_{h'}$ as above. Omitting the subscript $u$ for notational simplicity, let us write
\[
S(t)=\vierkanttwee{S_\ell(t)}{S_\ell^{\text{off}}(t)}{S_\ell^{\text{off}}(t)}{S_\ell^\perp(t)}: (\R^{\dim M}\oplus\Hil^\ell)\oplus\Hil^\ell_\perp\to (\R^{\dim M}\oplus\Hil^\ell)\oplus\Hil^\ell_\perp
\]
where the off-diagonal terms $S_\ell^{\text{off}}$ are equal because $S$ is symmetric. Similarly, we write $D=D_u$ as
\begin{eqnarray*}
D=\overline{\partial}+S=\vierkanttwee{\overline{\partial}+S_\ell}{S_\ell^{\text{off}}}{S_\ell^{\text{off}}}{\overline{\partial}+S_\ell^\perp}: &&\wt{H}^k_\phi(\R^2,\R^{\dim M}\oplus\Hil^\ell)\oplus \wt{H}^k_\phi(\R^2,\Hil^\ell_\perp)\\&&\to H^{k-1}_\phi(\R^2,\R^{\dim M}\oplus\Hil^\ell)\oplus H^{k-1}_\phi(\R^2,\Hil^\ell_\perp),
\end{eqnarray*}
where we use the abbreviation $\wt{H}^k_\phi(\R^2,\R^{\dim M}\oplus\Hil^\ell)=H^k(\R\times S^1,\R^{\dim M})\oplus\wt{H}^k_\phi(\R^2,\Hil^\ell)$. In order to prove the lemma, we want to show that $\op{\wt{\xi}}_{k,\sim}\leq c\op{D\wt{\xi}}_{k-1}$ for some constant $c$. Write $\wt{\xi}=\wt{\xi}_\ell+\wt{\xi}_\ell^\perp$ with $\wt{\xi}_\ell\in\wt{H}^k_\phi(\R^2,\R^{\dim M}\oplus\Hil^\ell)$ and $\wt{\xi}_\ell^\perp\in\wt{H}^k_\phi(\R^2,\Hil^\ell_\perp)$. Since on the finite-dimensional space $(\R^{\dim M}\oplus)\Hil^\ell$ the norms $\op{\cdot}_{k,\sim}$ and $\op{\cdot}_k$ are actually equivalent, we get using the definition of the norm $\op{\cdot}_{k,\sim}$ that
\[
\op{\wt{\xi}}_{k,\sim}&=\op{\wt{\xi}_\ell}_{k,\sim}+\op{\wt{\xi}_\ell^\perp}_{k,\sim}\\
&\leq c_1\op{\wt{\xi}_\ell}_{k}+\op{\overline{\partial}\wt{\xi}_\ell^\perp}_{k-1}.
\]
Since $i\partial_t+S$ is bounded with bounded inverse, we show below how one can use results from finite-dimensional Floer theory to deduce that
\begin{itemize}
    \item[$(*)$] for $\ell$ large enough $D_\ell=\CR+S_{\ell}:H^k_\phi(\R,\R^{\dim M}\oplus\Hil^\ell)\to H^{k-1}_\phi(\R,\R^{\dim M}\oplus\Hil^\ell)$ is an isomorphism.
\end{itemize}
Continuing the above estimates, we hence find that for $\ell$ large enough
\[
\op{\wt{\xi}}_{k,\sim}&\leq c_2 \pa{\op{D_\ell\wt{\xi}_\ell}_{k-1}+\op{\overline{\partial}\wt{\xi}_\ell^\perp}_{k-1}}
\]
Assuming for now that, for $\ell$ possibly larger, we have
\[
2\op{S_\ell^{\text{off}}\wt{\xi}_\ell}_{k-1}\stackrel{(**)}{\leq}\op{D_{\ell}\wt{\xi}_\ell}_{k-1}
\]
and
\[
2{\op{S_\ell^{\text{off}}\wt{\xi}_\ell^\perp}_{k-1}+2\op{S_\ell^\perp\wt{\xi}_\ell^\perp}_{k-1}}\stackrel{(***)}{\leq}\op{\CR\wt{\xi}_\ell^\perp}_{k-1}
\]
it then follows that
\begin{eqnarray*}
\op{\wt{\xi}}_{k,\sim}&\leq& 2c_2\pa{\op{D_{\ell}\wt{\xi}_\ell}_{k-1}-\op{S_\ell^{\text{off}}\wt{\xi}_\ell}_{k-1}}+2c_2\pa{\op{\CR\wt{\xi}_\ell^\perp}_{k-1}-\op{S_\ell^{\text{off}}\wt{\xi}_\ell^\perp}_{k-1}-\op{S_\ell^\perp\wt{\xi}_\ell^\perp}_{k-1}}\\
&=& 2c_2\pa{\op{D_{\ell}\wt{\xi}_\ell}_{k-1}-\op{S_\ell^{\text{off}}\wt{\xi}_\ell^\perp}_{k-1}}+2c_2\pa{\op{\CR\wt{\xi}_\ell^\perp}_{k-1}-\op{S_\ell^{\text{off}}\wt{\xi}_\ell}_{k-1}-\op{S_\ell^\perp\wt{\xi}_\ell^\perp}_{k-1}}\\
&\leq& 2c_2\pa{\op{\left(D\wt{\xi}\right)_\ell}_{k-1}+\op{\left(D\wt{\xi}\right)_\ell^\perp}_{k-1}}=2c_2\op{D\wt{\xi}}_{k-1}
\end{eqnarray*}
which proves the lemma.\\\par
Let us now prove that $(*)$, $(**)$ and $(***)$ hold for $\ell$ large enough. First note that since the orbit $u$ is nondegenerate up to small divisors, we have that \[\op{\xi}_{k,\sim}<c_2\op{(i\partial_t+S)\xi}_{k-1}\,\,\text{for some}\,\,c_2>0.\] Since $S^\ell\to S$ in the operator norm for operators $\wt{H}^k_\phi(\R,\R^{\dim M}\oplus\Hil_{h'})\to H^{k-1}_\phi(\R,\R^{\dim M}\oplus\Hil_{h'})$ by \Cref{findimest}, it follows that, for $\ell$ large enough, \[\op{\xi}_{k,\sim}< c_2\op{(i\partial_t+S_\ell)\xi}_{k-1},\,\,\text{and hence}\,\,\op{\xi_\ell}_{k,\sim}< c_2\op{(i\partial_t+S_\ell)\xi_\ell}_{k-1}.\] Using that the norms $\op{\cdot}_{k,\sim}$ and $\op{\cdot}_k$ are equivalent on $\Hil^\ell$, $(*)$ follows using \cite[Prop. 8.7.3]{AD}.\par In order to prove $(**)$ we use that the above argument can even be strenghtened as follows: It follows from \Cref{contained-1} that for every $h_0<h''<h$ there is some $c_3>0$ such that  \[\op{\xi_\ell}_{k,\sim}\geq c_3\ell^{-h''}\op{\xi_\ell}_k,\,\,\text{and hence}\,\,\op{\xi_\ell}_{k}\leq \ell^{h''}c_4\op{(i\partial_t+S_\ell)\xi_{\ell}}_{k-1}\] with $c_4=c_2c_3>0$. But now we can use \cite[Prop. 8.7.3]{AD} to deduce that
\[
\op{\wt{\xi}_\ell}_k\leq \ell^{h''}c_5\op{D_\ell\wt{\xi}_\ell}_{k-1}\,\,\textrm{for some $c_5>0$ independent of $\ell$.}
\]
Since $\op{S-S_\ell}=o(\ell^{-2h+1})$ as $\ell\to\infty$ by \Cref{findimest} when taking the norm for operators on $H^k_\phi(\R^2,\R^{\dim M}\oplus\Hil_{h'})$ and $2h-1>h>h'$ as $h>h_0\geq 2$, we conclude that one can indeed choose $\ell$ large enough such that $2\op{S_\ell^{\text{off}}\wt{\xi}_\ell}_{k-1}\leq\op{D_\ell\wt{\xi}_\ell}_{k-1}$, that is, we prove that the inequality $(**)$ holds. \par
To prove $(***)$, we again use \Cref{findimest} to conclude that for $\ell$ large enough, we have \[2\op{(S-S_{\ell})\wt{\xi}}_{k-1}\leq\op{\wt{\xi}}_{k,\sim}=\op{\overline{\partial}\wt{\xi}}_{k-1}\] and hence 
\[2\op{S\wt{\xi}_\ell^\perp}_{k-1}=2\op{(S-S_{\ell})\wt{\xi}_\ell^\perp}_{k-1}\leq\op{\overline{\partial}\wt{\xi}_\ell^\perp}_{k-1}\]
which implies $(***)$.

\end{proof}


With this we can give the proof of the Fredholm property. 

\begin{proof} \emph{(of \Cref{Fredholm})} 
Again we make use of the splitting $\Hil_{h'}=\Hil^\ell\oplus\Hil^\ell_\perp$ with $\Hil^\ell=\Span_{|n|\leq\ell}\{|z_n|_{h'}^{-1}z_n\}\subset\Hil_{h'}$ and write every $\wt{\xi}$ as a sum  
$\wt{\xi}=\wt{\xi}_\ell+\wt{\xi}_\ell^\perp$ with $\wt{\xi}_\ell\in\wt{H}^k_\phi(\R^2,\R^{\dim M}\oplus\Hil^\ell)=H^k(\R\times S^1,\R^{\dim M})\oplus\wt{H}^k_\phi(\R^2,\Hil^\ell)$ and $\wt{\xi}_\ell^\perp\in\wt{H}^k_\phi(\R^2,\Hil^\ell_\perp)$. Define $S=S_{\wt{u}}$ and $D=D_{\wt{u}}=\overline{\partial}+S$ as well as the asymptotic operators $S^\pm=S_{\wt{u}}^\pm=S_{u^{\pm}}$ and $D^\pm=D_{\wt{u}}^\pm=\overline{\partial}+S^\pm$. First use the proof of \Cref{injectivityestimate} to fix $\ell\in\N$ such that $(*)$ and $(**)$ holds for $D^+$ and for $D^-$, that is, $D^{\pm}_\ell=\CR+S^{\pm}_{\ell}:H^k_\phi(\R,\R^{\dim M}\oplus\Hil^\ell)\to H^{k-1}_\phi(\R,\R^{\dim M}\oplus\Hil^\ell)$ is an isomorphism and   
\[
2\op{S_\ell^{\pm,\text{off}}\wt{\xi}_\ell}_{k-1}\leq\op{D^\pm_{\ell}\wt{\xi}_\ell}_{k-1}.\]
On the other hand, by possibly choosing $\ell$ larger, it again directly follows from \Cref{findimest} that $(***)$ holds true for $D=\CR+S$,
\[2{\op{S_\ell^{\text{off}}\wt{\xi}_\ell^\perp}_{k-1}+2\op{S_\ell^{\perp}\wt{\xi}_\ell^\perp}_{k-1}}\leq\op{\CR\wt{\xi}_\ell^\perp}_{k-1}.
\]
For every $S_0>0$ define a cut-off function $\beta=\beta_{S_0}:\R\to[0,1]$ which equals one on $[-S_0+1,S_0-1]$ and zero outside $[-S_0,S_0]$. Since $S_{u}^\pm=\lim_{s\to\pm\infty}S_{\wt{u}}(s,\cdot)$, it follows that for sufficiently large $S_0>0$ we can employ $(*)$ and $(**)$ to deduce that there exists $c_1>0$ with 
\[\op{(1-\beta)\wt{\xi}_\ell}_{k,\sim}\leq c_1\op{D_\ell(1-\beta)\wt{\xi}_\ell}_{k-1}\leq 2c_1\left(\op{D_\ell(1-\beta)\wt{\xi}_\ell}_{k-1}-\op{S_\ell^{\text{off}}(1-\beta)\wt{\xi}_\ell}_{k-1}\right).\] Together with 
\begin{eqnarray*}
\op{\beta\wt{\xi}_\ell}_{k,\sim}=\op{\CR\beta\wt{\xi}_\ell}_{k-1}&\leq&\op{D_\ell\beta\wt{\xi}_\ell}_{k-1}+\op{S_\ell\beta\wt{\xi}_\ell}_{k-1}+2\op{S_\ell^{\text{off}}\beta\wt{\xi}_\ell}_{k-1}-2\op{S_\ell^{\text{off}}\beta\wt{\xi}_\ell}_{k-1}\\&\leq& 2\op{D_\ell\beta\wt{\xi}_\ell}_{k-1}+c_2\op{\beta\wt{\xi}_\ell}_{k-1} -2\op{S_\ell^{\text{off}}\beta\wt{\xi}_\ell}_{k-1},
\end{eqnarray*}
we get as in finite-dimensional Floer theory that there exists $c_3>0$ with 
\begin{eqnarray*}
\op{\wt{\xi}_\ell}_{k,\sim}&\leq& \op{\beta\wt{\xi}_\ell}_{k,\sim}+\op{(1-\beta)\wt{\xi}_\ell}_{k,\sim}\\&\leq& c_3\left(\op{D_\ell\wt{\xi}_\ell}_{k-1}+\op{\beta\wt{\xi}_\ell}_{k-1} -\op{S_\ell^{\text{off}}\wt{\xi}_\ell}_{k-1}\right).
\end{eqnarray*}
Together with $(***)$, summarizing we find similar as in the proof of \Cref{injectivityestimate} that there exists $c_4>0$ with
\begin{eqnarray*}
\op{\wt{\xi}}_{k,\sim}&\leq& c_4\pa{\op{D_{\ell}\wt{\xi}_\ell}_{k-1}+\op{\beta\wt{\xi}_\ell}_{k-1}-\op{S_\ell^{\text{off}}\wt{\xi}_\ell}_{k-1}}+c_4\pa{\op{\CR\wt{\xi}_\ell^\perp}_{k-1}-\op{S_\ell^{\text{off}}\wt{\xi}_\ell^\perp}_{k-1}-\op{S_\ell^\perp\wt{\xi}_\ell^\perp}_{k-1}}\\
&\leq& c_4\pa{\op{D_{\ell}\wt{\xi}_\ell}_{k-1}+\op{\beta\wt{\xi}_\ell}_{k-1}-\op{S_\ell^{\text{off}}\wt{\xi}_\ell^\perp}_{k-1}}+c_4\pa{\op{\CR\wt{\xi}_\ell^\perp}_{k-1}-\op{S_\ell^{\text{off}}\wt{\xi}_\ell}_{k-1}-\op{S_\ell^\perp\wt{\xi}_\ell^\perp}_{k-1}}\\
&\leq& c_4\pa{\op{\left(D\wt{\xi}\right)_\ell}_{k-1}+\op{\beta\wt{\xi}_\ell}_{k-1}+\op{\left(D\wt{\xi}\right)_\ell^\perp}_{k-1}}
=c_4\pa{\op{D\wt{\xi}}_{k-1}+\op{\beta\wt{\xi}_\ell}_{k-1}}
\end{eqnarray*}
as desired. Now \cite[proposition 8.7.4]{AD} shows that $D_{\wt{u}}$ is semi-Fredholm, meaning it has a finite-dimensional kernel and a closed range. \par
To show that $D_{\wt{u}}$ also has a finite-dimensional cokernel, the same line of proof as in \cite[section 8.7.c]{AD} or \cite[section 2.3]{Sa} holds: Since the adjoint operator $D_{\wt{u}}^*$ of $D_{\wt{u}}$ is given by 
\[
D_{\wt{u}}^*=\overline{\partial}-S_{\wt{u}}^*:\wt{H}^k_{\phi}(\R^2,\R^{\dim M}\oplus\Hil_{h'})\to H^{k-1}_\phi(\R^2,\R^{\dim M}\oplus\Hil_{h'}),
\]
it immediately follows that $D_{\wt{u}}^*$ is also semi-Fredholm. On the other hand, since the kernel and cokernel of $D_{\wt{u}}$ and $D_{\wt{u}}^*$ are independent of $1< k\leq\floor{h/d}$ and $h_0 < h'<h$ by \Cref{linear-regularity}, it follows that the cokernel of $D_{\wt{u}}:\wt{H}^k_{\phi}(\R^2,\R^{\dim M}\oplus\Hil_{h'})\to H^{k-1}_\phi(\R^2,\R^{\dim M}\oplus\Hil_{h'})$ agrees with the kernel of $D_{\wt{u}}^*:\wt{H}^k_{\phi}(\R^2,\R^{\dim M}\oplus\Hil_{h'})\to H^{k-1}_\phi(\R^2,\R^{\dim M}\oplus\Hil_{h'})$. \end{proof}

\section{Appendix: Nondegeneracy up to small divisors is generic}
\label{transversalitysection}
Consider the set $\mathcal{G}_h:=C^{\floor{h/d}}_{\phi}(\R,C^{\floor{h/d}+2}(M\times\Hil_{-h},\R))$ of $h$-regularizing Hamiltonians $G_t=F_t\circ\phi^A_{-t}$ as in \Cref{admissible-2}, that is, $G_t\in C^{\floor{h/d}+2}(M\times\Hil_{-h},\R)$ for every $t\in\R$ with $G_{t+T}=G_t\circ\phi^A_{-T}$ and the map $t\mapsto G_t$ is $\floor{h/d}$-times continuously differentiable. 
In this section we show that, for a generic choice of $G_t\in \mathcal{G}_h$, for each $u\in\P(\phi^A_T,G)$ the linearization $L_u$ of $u\mapsto J(u)\partial_t u+\nabla G_t(u)$ at $u$ is surjective. This implies that for generic choice of $G\in\mathcal{G}_h$ the orbit space $\P(\phi^A_T,G)$ is a manifold of dimension $0$ and all orbits are non-degenerate up to small divisors. 
The outline of the proof is similar to standard proofs of transversality in Floer theory, see e.g. \cite{SZ}. Namely, we consider the universal orbit space $\wt{\P}(\phi^A_T)$ which is the union of all orbits spaces for all choices of $G\in\mathcal{G}_h$. We show that the linearization of the map $({u},G)\mapsto{i\partial_t}u+\nabla G_t({u})$ is surjective to conclude that the universal orbit space is transversally cut out and then use Sard-Smale to show that the set of regular values of the projection map from the universal orbit space onto the space of nonlinearities is dense. \par 

After this quick summary, let us list the key results of this section. 
\begin{theorem}
\label{transversalitythm}
The subset of nonlinearity Hamiltonians $G\in\mathcal{G}_h=C^{\floor{h/d}}_{\phi}(\R,C^{\floor{h/d}+2}(M\times\Hil_{-h},\R))$ for which the linearized operator 
\[
L_u:\wt{H}_\phi^k(\R,\R^{\dim M}\oplus\Hil_{h'})\to H^{k-1}_\phi(\R,\R^{\dim M}\oplus\Hil_{h'})
\]
is surjective for all ${u}\in\P(\phi^A_T,G)$, is of the second Baire category, that is, a countable intersection of open and dense subsets.
\end{theorem}
Note that here we again use the abbreviation $\wt{H}^k_\phi(\R,\R^{\dim M}\oplus\Hil_{h'})=H^k(S^1,\R^{\dim M})\oplus\wt{H}^k_\phi(\R,\Hil_{h'})$. By applying the Implicit Function Theorem, we get the following important 

\begin{corollary}
For a generic choice of $G\in\mathcal{G}_h$ all $\phi^A_T$-periodic orbits of $G_t$ are nondegenerate up to small divisors. Furthermore the orbit set $\P(\phi^A_T,G)=\pa{i\partial_t+\nabla {G}}^{-1}(0)$ is a smooth manifold of dimension $0$, that is, a discrete set of points in $H^k(S^1,M)\times\wt{H}^k_\phi(\R,\Hil_{h'})$.
\end{corollary}

Let us define the universal orbit space 
\[
\wt{\P}^{\floor{h/d}}(\phi^A_T)=\bigcup_{G\in \mathcal{G}_h}\wt{\P}(\phi^A_T,G)\subset H^k(S^1,M)\times\wt{H}^k_\phi(\R,\Hil_{h'})\times \mathcal{G}_h
\]
whose elements are pairs $({u},G)$ of orbits ${u}$ for choice of nonlinearity $G$. The proof of \Cref{transversalitythm} mostly relies on the proof of the following lemma. 
\begin{lemma}
\label{lemma3}
The linearization of the map $({u},G)\mapsto {i\partial_t}u+\nabla{G}_t({u})$ at the point $({u},G)\in\wt{\P}^{\floor{h/d}}(\phi^A_T)$, given by
\begin{eqnarray*}
L_{{u},G}:&&\wt{H}^k_\phi(\R,\R^{\dim M}\oplus\Hil_{h'})\oplus C^{\floor{h/d}}_{\phi}(\R,C^{\floor{h/d}+2}(M\times\Hil_{-h},\R))\to H^{k-1}_\phi(\R,\R^{\dim M}\oplus\Hil_{h'}),\\&&(\xi,K)\mapsto i\partial_t\xi+S_u\xi+\nabla K_t({u})
\end{eqnarray*}
is surjective for all $({u},G)\in\wt{\P}^{\floor{h/d}}(\phi^A_T)$. 
\end{lemma}
\begin{proof}
The proof is a small adaption of the original proof for finite-dimensional symplectic manifolds from \cite{SZ}. Fix some $\eta\in H^{k-1}_\phi(\R,\R^{\dim M}\oplus\Hil_{h'})$ and assume 
\begin{align}
\label{assumption1}
\langle L_{{u},G}(\xi,K),\eta\rangle=0\quad \text{for all\;}(\xi,K)\in \wt{H}^k_\phi(\R,\R^{\dim M}\oplus\Hil_{h'})\oplus C^{\floor{h/d}}_{\phi}(\R,C^{\floor{h/d}+2}(M\times\Hil_{-h},\R)).
\end{align}
We show that this implies $\eta\equiv0$ which in turn shows that $L_{{u},G}$ is surjective.\par
Note that \eqref{assumption1} implies that both $\langle L_u\xi,\eta\rangle=0$ for all $\xi\in\wt{H}^k_\phi(\R,\R^{\dim M}\oplus\Hil_{h'})$ and $\langle \nabla K_t({u}),\eta\rangle=0$ for all $K\in \mathcal{G}_h$. Since $L_u$ is self-adjoint, it follows that $\eta\in\ker L_u$ so that we have $\eta\in \wt{H}^{\floor{h/d}+1}_\phi(\R,\R^{\dim M}\oplus\Hil_h)\subset H^{\floor{h/d}+1}_\phi(\R,\R^{\dim M}\oplus\Hil)\subset C^{\floor{h/d}}_\phi(\R,\R^{\dim M}\oplus\Hil)$. The second equality $\langle\nabla K_t({u}),\eta\rangle=0$ for all $K$ now implies that $\eta\equiv0$: 
Indeed, assume that there were some $t_0\in\R$ such that $\eta(t_0)\neq 0$, where we assume without loss of generality that $\eta(t_0)>0$. Note first that we can clearly find $K\in C^{\floor{h/d}}_{\phi}(\R,C^{\floor{h/d}+2}(M\times\Hil_{-h},\R))$ such that $\nabla K_{t_0}({u}(t_0))\cdot\eta(t_0)>0$. Note however that, while $\eta(t_0)\in\Hil$, we necessarily have $\nabla K_{t_0}(u(t_0))\in\Hil_{+h}$. It follows that in our infinite dimensional situation \emph{we now additionally need to use that the dense subspace $\Hil_{+h}$ has no complement in $\Hil$.} Since $\eta$ is $\floor{h/d}$-times continuously differentiable, there is a neighborhood of $t_0$ where $\eta$ is positive. Using a smooth cut-off function with support in the image of this neighborhood under the map $u$, we see that we can find $K\in C^{\floor{h/d}}_{\phi}(\R,C^{\floor{h/d}+2}(M\times\Hil_{-h},\R))$ such that $\langle \nabla K_t({u}),\eta\rangle >0$, so that $\eta\equiv 0$ by proof by contradiction.  \end{proof}

 \begin{proof}[Proof of \Cref{transversalitythm}]
Using the Sard-Smale theorem from \cite{Sm} it follows that the set of regular value of the projection map $\pi:\wt{\P}^{\floor{h/d}}(\phi^A_T)\to C^{\floor{h/d}}_{\phi}(\R,C^{\floor{h/d}+2}(M\times\Hil_{-h},\R))$ is of the second category in the sense of Baire. When $G$ is a regular value of the projection map $\pi$, then it is immediate to see that $L_{{u}}$ is surjective for all ${u}\in\P(\phi^A_T,G)$.
\end{proof}



\end{document}